\documentclass[preprint,11pt,oneside]{elsarticle}
\usepackage[english]{babel}
\usepackage[T1]{fontenc}
\usepackage{graphicx}
\usepackage{bm}
\usepackage{amsmath,amsfonts,amssymb,amsthm}
\usepackage{latexsym}
\usepackage{booktabs}
\usepackage{tikz}
\usepackage{multirow}
\usepackage[TABBOTCAP]{subfigure}
\usepackage{cases}
\usepackage{setspace} 
\usepackage{pifont}
\usepackage{tabularx}
\usepackage[titletoc,title]{appendix}
\numberwithin{equation}{section}
\usepackage[bookmarks,colorlinks=true,pdfauthor={Cinzia Soresina},linkcolor=red,citecolor=red]{hyperref}

\def\be{\begin{equation}}
\def\ee{\end{equation}}
\def\be{\begin{equation}}
\def\ee{\end{equation}}

\newtheorem{lm}{\bf Lemma}

\newtheorem{tr}{\bf Theorem}
\newtheorem{cor}{\bf Corollary}

\theoremstyle{remark}
\newtheorem{rem}{\bf Remark}
\definecolor{inchworm}{rgb}{0.7, 0.93, 0.36}

\usepackage{xcolor}
\definecolor{blun}{cmyk}{0.8, 0.5, 0, 0.7}
\definecolor{darkpastelred}{rgb}{0.76, 0.23, 0.13}
\definecolor{bostonuniversityred}{rgb}{0.8, 0.0, 0.0}

\let\today\relax
\makeatletter
\def\ps@pprintTitle{%
    \let\@oddhead\@empty
    \let\@evenhead\@empty
    \def\@oddfoot{\footnotesize\itshape
         {Submitted preprint} \hfill\today}%
    \let\@evenfoot\@oddfoot
    }
\makeatother

\begin{document}
\let\oldproofname=\proofname
\renewcommand{\proofname}{\rm\bf{\oldproofname}}

%

\title{Persistence in seasonally varying predator--prey systems\\ with Allee effect}
\author[cmaf]{C. Rebelo}
\ead{mcgoncalves@fc.ul.pt}
\author[tum]{C. Soresina}\corref{cor1}
\ead{soresina@ma.tum.de}

\cortext[cor1]{Corresponding author}
\address[cmaf]{CMAFcIO Centro de Matem\'atica, Aplica\c{c}\~{o}es Fundamentais e Investiga\c{c}\~ao Operacional,\\ Faculdade de Ci\^encias, Universidade de Lisboa, Edificio C6, Campo Grande, 1749-016 Lisboa, Portugal}
\address[tum]{Zentrum Mathematik, Technische Universit\"at M\"unchen,\\ Boltzmannstr. 3, 85748 Garching bei M\"unchen, Germany}
\journal{xxxxxx}
\begin{abstract}
A generalized seasonally-varying predator--prey model with Allee effect in the prey growth is investigated. The analysis is performed only on the basis of some properties determining the
shape of the prey growth rate and the trophic interaction functions. General conditions for coexistence are determined, both in the case of weak and strong Allee effect. Finally, a modified Leslie--Gower predator--prey model with Allee effect is investigated. Numerical results illustrate the qualitative behaviors of the system, in particular the presence of periodic orbits.
\end{abstract}


\maketitle

\section{Introduction}


More than 90 years after their introduction \cite{volterra1926variazioni, Volterra}, Lotka-Volterra predator-prey models are a topic which attracts the attention of many researchers  from a large range of points of view. Concerning recent works, we mention examples of results on predator-prey models dealing with autonomous ODEs \cite{vera2019dynamics, gonzalez2019competition}, with delay \cite{liu2018bifurcation, chen2019spatiotemporal} and stochastic  \cite{liu2019dynamics, upadhyay2019global} systems of equations,  time-discrete  \cite{weide2019hydra, khan2019bifurcations}, network  \cite{weng2019predator, upadhyay2019spectral} and fractional  \cite{owolabi2018modelling} models.
Cross diffusion and pattern formation are also topics with recent results \cite{conforto2018, desvillettes2019, tulumello2014cross} as well as the derivation of classical functional responses \cite{metz2014dynamics, geritz2012mechanistic, dawes2013derivation} and the impact of seasonality in ODEs models \cite{lisena2018global, lopez2019multiplicity}.

Seasonally varying predator--prey models, described as non autonomous ODEs systems depending on periodic coefficients in order to take into account the seasonal changes of the environment in which the predation process takes place \cite{basille2013ecologically}, have not been deeply investigated as other types of predator--prey models.

As far as we know the first study of predator--prey models with periodic coefficients was the paper \cite{cushing1977periodic} by Cushing. It studied the existence of periodic solutions in a non-autonomous predator--prey model by use of standard  techniques of bifurcation theory. Later, other authors have investigated the existence of periodic solutions, the persistence and chaos in seasonally models \cite{amine1994periodic, bohner2006existence, cui2006permanence, fan2003dynamics, kuznetsov1992, lopez1996periodic, rinaldi1993multiple, zhidong1999, zhidong1999uniform}. 

In a recent paper \cite{garrione2016} a seasonally dependent predator--prey model with general growth rate for the prey and in which the functional response can belong to a large class of functions was considered and persistence for the predators and for the prey was obtained, extending the notion of basic reproductive number $R_0$ to this context \cite{rebelo2012}, concept originally introduce in epidemic models. Here we denote with $R_0$ \cite{georgescu07,garrione2016} the basic reproduction number, i.e. the number of predators a predator gives rise during its life when introduced in a prey-population. Thus, using the results in \cite{zhao2003} the existence of a non-trivial periodic solution was guaranteed when $R_0>1$.  This work generalizes several previous papers in what concerns persistence and also because it gives persistence results for a  very general model class which include two prey-two predators models, Leslie--Gower models. The results were obtained using a technique based on an abstract theorem given in \cite{fonda88} and already applied in \cite{rebelo2012,garrione2016}. 

Nevertheless, even if the growth rate considered for the prey is quite general, this class of models does not include the Allee effect \cite{allee1949principles, stephens1999allee} (observed in populations of bisexual organisms and/or with a team behavior and a mutual help). In fact, to model the Allee effect, the prey growth function is not monotonic but it increases for small population abundance. The Allee effect can be weak or strong, depending on the sign of prey growth function (non-negative or negative, respectively) for small population abundance. Autonomous predator--prey models with Allee effect in the prey are largely analyzed in literature \cite{buffoni2011, buffoni2016, gonzalez2011multiple, van2007heteroclinic, wang2011predator}, the analogous problem for seasonally dependent models was, as far as we know, less studied. 

There are few studies concerning population models with the Allee effect in which seasonality is considered. With respect to models describing the growth of one single species, we refer the papers \cite{padhi2010,rizaner12}. In these papers the growth of the species is modeled by the equation 
$$y'=a(t)y(y-b(t))(c(t)-y)$$
where $a(t)$, the intrinsic growth of the species, $b(t)$, its Allee threshold and $c(t)$, the carrying capacity of the habitat, are all seasonally dependent functions. In the first paper the condition $b(t)<c(t)$ for each $t$ was assumed but the case $\inf c<\max b$ is allowed. Under some additional conditions a result about the existence of at least two non-trivial periodic solutions for the equation was obtained using the Leggett--Williams multiple fixed point theorem on cones \cite{leggett1979}. In \cite{rizaner12} the case $\max b< \inf c$ was analysed, and the existence of two nontrivial periodic solutions was guaranteed, also establishing their stability properties.

In this paper we consider a general predator--prey model with seasonality in which we take into account an Allee effect on the prey growth. Both cases of weak and strong Allee effects are analyzed. The keynote point of the paper is that the theoretical results are obtained for a general class of models, only on the basis of some properties determining the shape of the prey growth function and of the functional response. In the case a weak Allee effect is considered, we prove extinction when the basic reproduction number $R_0<1$, persistence when $R_0>1$, and the existence of a periodic solution when $R_0>1$. The results are obtained using the technique in \cite{garrione2016} after some preliminary steps. In the strong Allee effect case, in order to prove the main theorem, an auxiliary result (see Theorem \ref{strong_preyonly}) about the existence of two non-trivial periodic solutions $N_\pm^*(t)$ in the case of a seasonally dependent model for the evolution of one species is obtained. This result generalizes analogous one stated in \cite{rizaner12}.  Thanks to this auxiliary result, we are able to prove extinction of the predators if $\lambda^-_2<\lambda^+_2<1$ and the existence of a nontrivial periodic solution if $\lambda^-_2<1<\lambda^+_2$, where 
$$\displaystyle \lambda^\pm_2=\int_0^T ( \gamma(t)f(t, N_\pm^*(t),0)-\delta_1(t)) dt,$$
being $T$ the period, $\gamma,\; \delta_1$ the conversion factor and the mortality rate of predators respectively, and $N_\pm^*(t)$ the $T$-periodic orbits in absence of predators.

Note that in the case of strong Allee effect there are two stable periodic orbits in the predator-free line: the origin and $N_+^*(t).$ We have that if $\lambda^+_2>1$ the same is verified by the basic reproduction number associated to $(N_+^*(t),0)$ but in this case persistence is not guaranteed by this condition.
The result on the existence of a periodic solution is based on degree theory following ideas in \cite{zhidong1999}, see also  \cite{ortega1995,amine1994periodic, alvarez1986application, makarenkov2014topological}. In the case $1<\lambda^-_2<\lambda^+_2$, numerical simulations lead us to conjecture that the only feature is the predator extinction, but this is still an open problem that we want to address in a future work. Finally, we point out that periodic predator--prey models of Leslie--Gower type can be treated using the same techniques. Also in this case, numerical simulations are reported in order to show the possible outcomes.

In Section \ref{Sec:Model} we describe the model and the conditions satisfied by the growth rate of the prey, both in the case of weak and  in the case of strong Allee effect. We also describe here the class of admissible functional responses. Section \ref{Sec:Weak} is dedicated to the case of weak Allee effect, while in Section \ref{Sec:Strong} the case of strong Allee effect is analyzed. The slightly different class of models of Leslie--Gower type is then treated in Section \ref{Sec:LG}. Finally in Section \ref{Sec:NumRes} some numerical simulations are reported which illustrate the obtained results and help us to make some conjectures. In Section \ref{Sec:Concl} some concluding remarks can be found.

\section{The model}\label{Sec:Model}
Indicating with $N,\;P$ the prey and predator abundance respectively, the system writes
\begin{equation}\label{PP}
\begin{cases}
\dot{N}=k(t,N)N-f(t,N,P)P,\\
\dot{P}=\gamma(t)f(t,N,P)P-\delta_1(t)P-\delta_2(t)P^2.
\end{cases}
\end{equation}
We assume that\footnote{Hereafter $\mathbb{R}_+$ denotes the set of non-negative real numbers. }
$$k:\mathbb{R}\times\mathbb{R}_+\to\mathbb{R},\quad f:\mathbb{R}\times\mathbb{R}_+^2\to\mathbb{R}_+,\quad \delta_i(t):\mathbb{R}\to\mathbb{R}_+,\;i=1,2,\quad \gamma:\mathbb{R}\to\mathbb{R}_+$$
are continuous functions, $T$-periodic ($T>0$) in the $t$-variable and continuous differentiable in $N,\;P$ (if depending on such variables). Here $k(t,N)N$ corresponds to the prey growth in absence of predators, $f(t,N,P)$ is the predator functional response and $\gamma(t)f(t,N,P)$ is the numerical response. The term $\delta_1(t)$ corresponds to the death rate of predators, while $\delta_2(t)P^2$ is related to an intra-specific competition between predators.

As a preliminary assumption, we ask that:
\begin{itemize}
	\item $\delta_1(t)>0$ for every $t\in[0,T]$;
	\item if $\delta_2(t) \not\equiv 0$, we assume that $\min_{t\in [0,T]}{\gamma(t)>0}$. This is actually a quite common assumption in literature.
\end{itemize}

We now introduce our main hypotheses on the terms $k(t,N)$ and $f(t,N,P)$. In order to simplify the notation hereafter, for each periodic function $z:\mathbb{R}\to \mathbb{R}_+$ we set
$$\underline{z}=\min_{t\in [0,T]}{z(t)}, \qquad \overline{z}=\max_{t\in [0,T]}{z(t)}.$$
\medskip

\subsection{The prey growth function}\label{Sub:PreyGrowth}
We first deal with the prey growth function, assuming some properties to describe an Allee effect. We distinguish two cases, a weak Allee effect and a strong Allee effect.\\[0.3cm]

{\bf Weak Allee effect}\\[-0.5cm]
\begin{description}
	\item[(gw1)] for every $t$, there exists $K_+(t)>0$ such that $k(t,K_+(t))=0$,
	\item[(gw2)] for every $t$, $k(t,0)>0$,
	\item[(gw3)] for every $t$, there exist $\xi(t) \in [0,\;\underline{K}_+]$ such that $\dfrac{\partial k}{\partial N}(t,\xi(t))=0$,
	\item[(gw4)] for every $t$, $\dfrac{\partial k}{\partial N}(t,N)(N-\xi(t))<0$, when $N\neq \xi(t)$.
\end{description}

\vspace{0.5cm}
{\bf Strong Allee effect}\\[-0.5cm]
\begin{description}
	\item[(gs1)] for every $t$, there exist $K_-(t),\;K_+(t)>0$ such that $k(t,K_-(t))=0,\; k(t,K_+(t))=0$ and $\overline{K}_-<\underline{K}_+$,
	\item[(gs2)] for every $t$, $k(t,0)<0$, and $k(t,N)(N-K_-(t))(K_+(t)-N)>0$,
	\item[(gs3)] for every $t$, there exist $\xi(t) \in (\overline{K}_-,\underline{K}_+)$ such that $\dfrac{\partial k}{\partial N}(t,\xi(t))=0$,
	\item[(gs4)] for every $t$, $\dfrac{\partial k}{\partial N}(t,N)(N-\xi(t))<0$, when $N\neq\xi(t)$.
\end{description}

The quantity $K_-(t)$ represents the minimum population size, and $K_+(t)$ is the carrying capacity of the habitat at time $t$.
Note that by the $T$-periodicity of $k$ in $t$,  $K_+(t)$ and $\xi(t)$ and eventually $K_-(t)$ are $T$-periodic. Also under these hypothesis $k(t,N)$ turns out to be bounded above. 
In case of a weak Allee effect, the first and the last assumptions imply that the function $k$, fixed $t$, has only one zero, $K_+(t)>0$ and that it is positive when $0 < x < K_+(t)$ and negative when $x >K_+(t)$, while in case of strong Allee effect the function $k$, fixed $t$, has only two zeroes, $K_-(t)$ and $K_+(t)>0$, that it is positive when $K_-(t)< N(t)< K_+(t)$ and negative when $0 < N(t) < K_-(t)$ or $N(t) > K_+(t)$. Note that property (gw4) implies (gw3), as well as (gs4) and (gs3); nevertheless we prefer to present the properties in this way for sake of clarity.

We are also assuming just one maximum when $k$ is positive, meaning that there is an optimal population size corresponding to a maximum rate of growth. Far from this optimal value the prey growth rate decreases: for values smaller than $\xi$ the individuals could be very few and too sparse, while for greater values than $\xi$ the competition for resources becomes evident.

Qualitative shapes of $k$ for weak and strong Allee effect are shown in Figure \ref{figG}.

\rem{With these general assumptions, which are satisfied by functions reported in \cite{buffoni2011,buffoni2016}, $k$ is suitable to model a weak and a strong Allee effect on the prey growth.  In \cite{buffoni2011,buffoni2016} other technical assumptions on the prey growth were considered, which involve the first and the second derivatives of $k$ (with respect to $N$), in order to perform the existence and stability analysis of equilibrium states.}

\begin{figure}
\begin{center}
\subfigure[\label{weak}]{\includegraphics[width=.5\columnwidth,trim={4.2cm 8cm 4.5cm 8cm},clip]{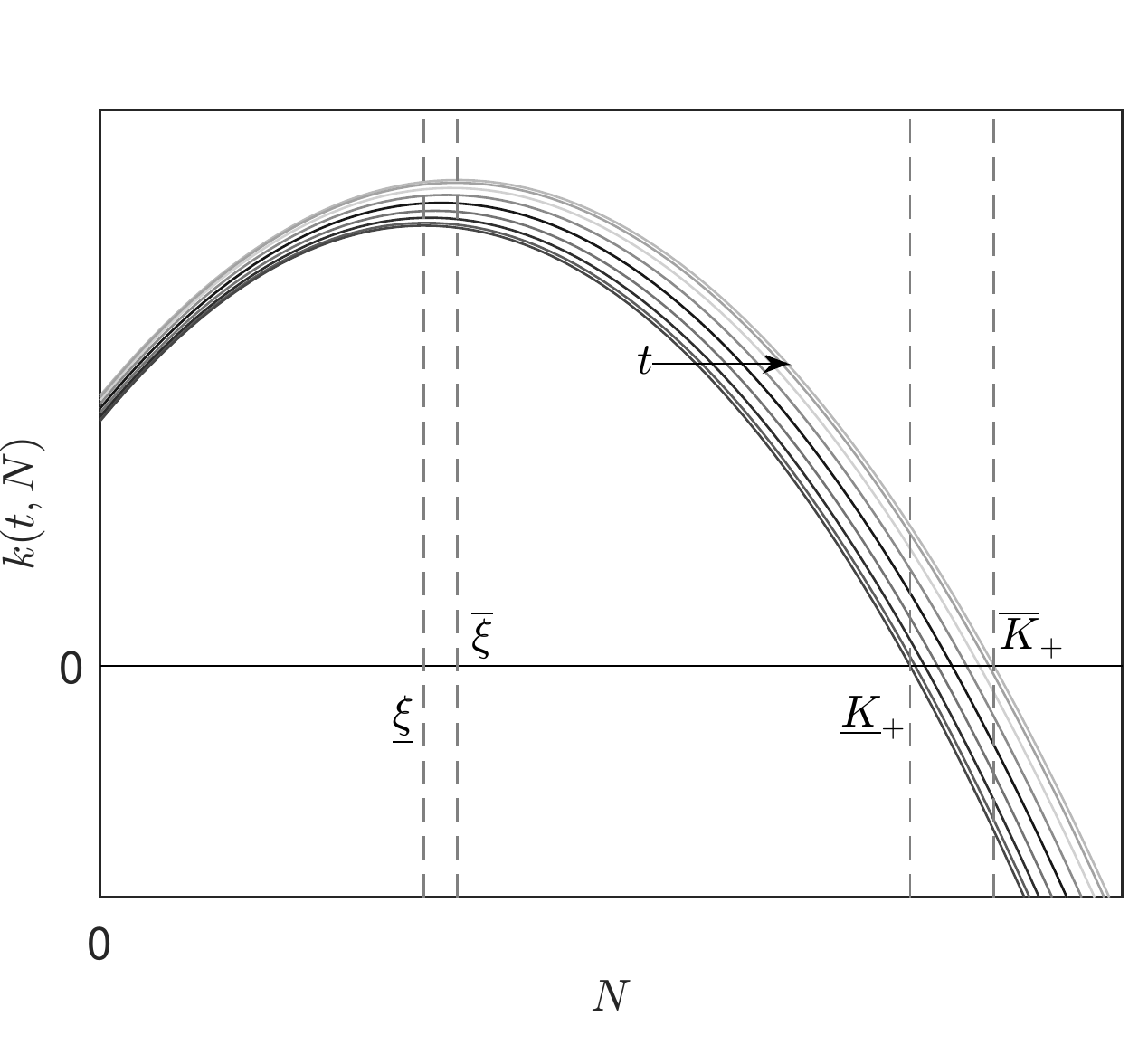}}
\hspace{-0.3cm}
\subfigure[\label{strong}]{\includegraphics[width=.5\columnwidth,trim={4.2cm 8cm 4.5cm 8cm},clip]{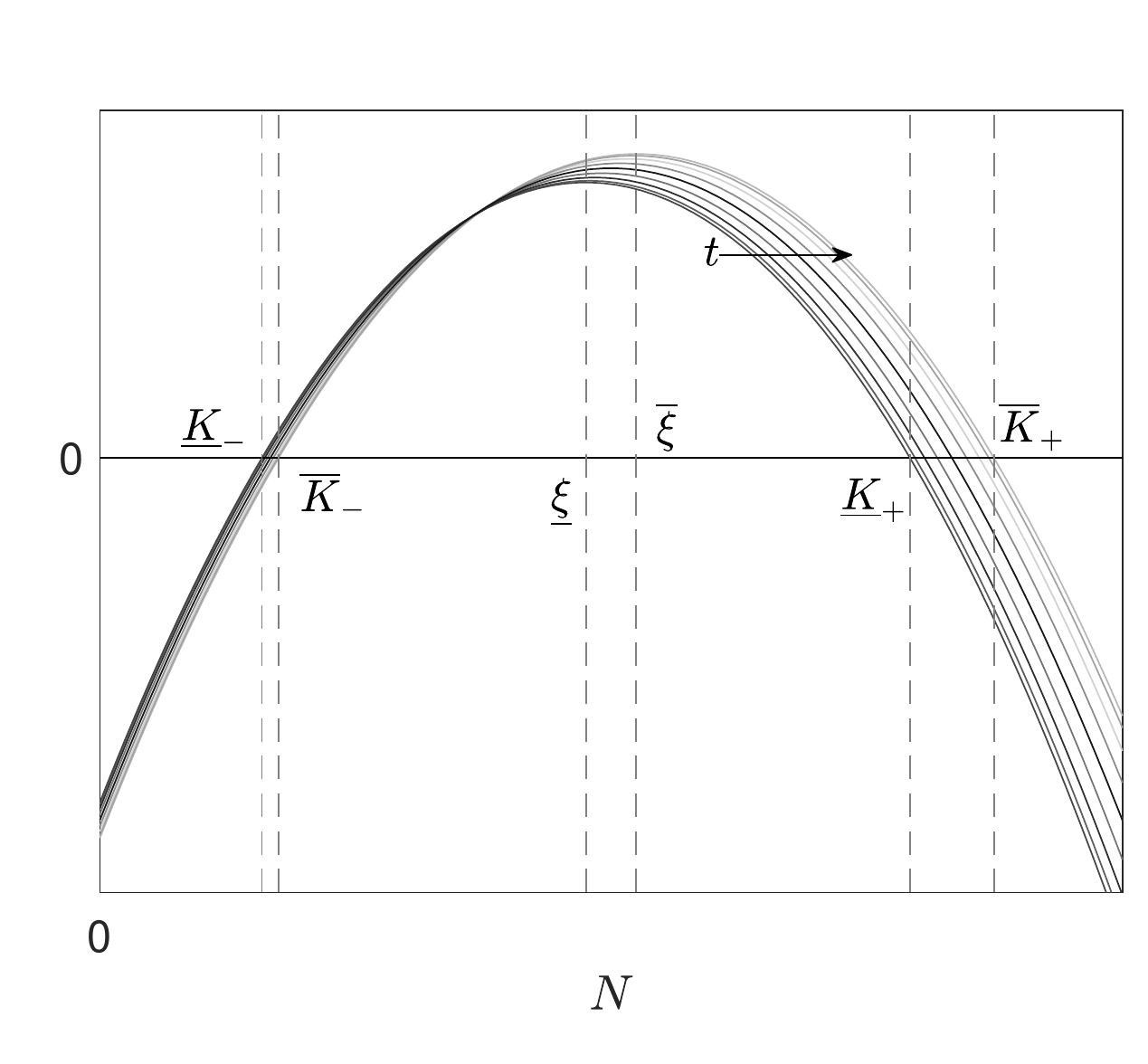}}
\caption{Qualitative shape of the function $k(t,N)$ in case of weak (a) and strong (b) Allee effect.}\label{figG}
\end{center}
\end{figure}

\subsection{The functional response}
The conditions on the functional response are the following:
\begin{description}
	\item[(f1)] for all $t$, $f(t,0,P)\equiv 0$,
	\item[(f2)] for every $t,\; P$, $f(t,N,P)>0$ if $N>0$,
	\item[(f3)] for all $t$, fixed $N$, the function $P\mapsto f(t,N,P)$ is non-increasing in $P$ (or independent of it),
	\item[(f4)] for all $t$, fixed $P$, the function $N\mapsto f(t,N,P)$ is non-decreasing in $N$,
	\item[(f5)] for all $t$, there exists a non negative continuous function $f_0(t,P)$ such that it holds
	      $$f_0(t,P)=\lim_{N\to 0^+}\dfrac{f(t,N,P)}{N},$$
				uniformly in $t\in \mathbb{R}$ and $P$ belonging to a compact set.
\end{description}

The first hypothesis imposes that predator functional response cannot be positive if there are no preys and the second implies that the functional response describes a loss term in the prey equation and a gain term in the predator one. The third one essentially says that the more the predators, the less one prey is expected to contribute to their growth; the fourth says that when there are more preys, each one contributes more to the growth of the predators. Finally, the fifth property is technical but reasonable; in particular it is used in the proof of Theorem \ref{teo_oer_weak}.

Functional responses for predator--prey models widely used in literature satisfying these properties are listed in \cite{garrione2016}.

\section{Persistence with a weak Allee effect}\label{Sec:Weak}

Assuming that the prey growth function describes a weak Allee effect (properties (gw1)--(gw4)), it is possible to obtain the following results, as in \cite{garrione2016}. However, the proof of Theorem \ref{th:wA0p} differs from the one in \cite{garrione2016}, where the assumption on the monotonicity of $k$ with respect to $N$ is crucial.

\begin{tr}\label{th:wA0p}
Assume that the prey growth function satisfies the listed properties for a weak Allee effect (properties (gw1)--(gw4)). Consider the prey dynamics in absence of predator described by 
\begin{equation}\label{weak_preyonly}
\dot{N}=k(t,N)N-\nu N,
\end{equation} 
where $\nu>0$ is constant. Then,  for every $0\leq \nu<\nu_*$ where
$$\nu*:=\dfrac{1}{T}\int_{0}^{T}k(s,0)ds,$$
equation \eqref{weak_preyonly} admits a unique, positive, bounded, globally asymptotically stable $T$-periodic solution $N^{*,\nu}$. Moreover, 
\begin{equation}\label{limN_weak}
\lim_{\nu\to 0}N^{*,\nu}(t)=N^{*,0}(t) \quad uniformly\; in\; t\in [0,T].
\end{equation}
\end{tr}

\begin{proof}
Our proof follows \cite[p.128--129]{hale2012}.

The trivial solution $N(t)\equiv 0$ is a $T$-periodic solution which, by the definition of $\nu$, is unstable. Hence there exists a positive solution $N(t)$ such that $\lim_{t\to +\infty} N(t)\neq 0$.

Moreover when $N>\overline{K}_+$ the right member of \eqref{weak_preyonly} is negative. Therefore $N(t)$ is bounded  in $ \mathbb{R}$ and  approaches a $T$-periodic solution $N^{*,\nu}(t)$ as $t \to +\infty$. Furthermore, this $T$-periodic solution has the property $\underline{K}_+^\nu<N^{*,\nu}(t)<\overline{K}_+^\nu$ and hence is positive and bounded.

To prove the uniqueness of this $T$-periodic solution, we suppose that $N^{*,\nu}_1(t)$ and $N^{*,\nu}_2(t)$ are two $T$-periodic solutions such that $v(t):=N^{*,\nu}_1(t) - N^{*,\nu}_2(t)>0$ for all $t\in[0,T]$. (Note that this hypothesis is not restrictive; if there exists $\hat{t}$ such that $N^{*,\nu}_1(\hat{t}) = N^{*,\nu}_2(\hat{t})=\hat{N}$, then $\hat{N}$ is an equilibrium point.)

Then we have
\begin{align*}
\dot{v}(t)&=\dot{N}^{*,\nu}_1(t)-\dot{N}^{*,\nu}_2(t)\\
              &=\left(k(t,N^{*,\nu}_1(t))-\nu\right) N^{*,\nu}_1(t) - \left(k(t,N^{*,\nu}_2(t))-\nu\right) N^{*,\nu}_2(t)\\
              &=k(t,N^{*,\nu}_1(t)) \left(N^{*,\nu}_1(t)-N^{*,\nu}_2(t)\right)+ \left( k(t,N^{*,\nu}_1(t)) -k(t,N^{*,\nu}_2(t))\right)N^{*,\nu}_2(t)+\\
& -\nu \left(N^{*,\nu}_1(t)-N^{*,\nu}_2(t)\right).
\end{align*}
Since $k(t,N)$ is decreasing in $[\xi(t),+\infty)$ with respect to $N$, we have that $ k(t,N^{*,\nu}_1(t)) -k(t,N^{*,\nu}_2(t))<0$, and then it turns out 
$$\dot{v}(t)<k(t,N^{*,\nu}_1(t))v-\nu v.$$
Since we have
$$\dfrac{\dot{v}}{v}<k(t,N^{*,\nu}_1(t))-\nu,$$
we obtain
$$v(T)<v(0)\exp{\left(\int_0^T{\left( k(s,N^{*,\nu}_1(s))-\nu\right)ds}\right)}.$$
But remembering that $N^{*,\nu}_1(t)$ is a $T$-periodic orbit, we have
$$\int_0^T{\left( k(s,N^{*,\nu}_1(s))-\nu\right)ds}=\log{\dfrac{N^{*,\nu}_1(T)}{N^{*,\nu}_1(0)}}=0,$$
from which we obtain $v(T)<v(0)$ and the contradiction.

Finally, since $N^{*,\nu}(t)$ is the only $T$-periodic solution with positive initial data, it is globally asymptotically stable.

We have finally to prove the validity of \eqref{limN_weak}. To this end, since the solutions $N^{*,\nu}$ satisfy \eqref{weak_preyonly} and they are bounded, thanks to the Gronwall's lemma they are uniformly bounded with respect to $\nu$. Thus, denoting by $w_\nu$ the corresponding fixed point of $P^T_\nu$, the $T$-Poincar\'e operator associated with \eqref{weak_preyonly}, we can assume without loss of generality that $w_\nu$ converges. In view of the uniqueness of the fixed point $w_\nu$ and the fact that $P^T_\nu$ depends continuously on $\nu$ when restricted to compact sets, we thus infer that $w_\nu \to w_0$, and the assertion \eqref{limN_weak} follows from the continuous dependence on the initial datum.
\end{proof}

\begin{tr}\label{teo_oer_weak}
Assume that the prey growth function and the functional response satisfy properties of a weak Allee effect (gw1)--(gw4) and (f1)--(f5), respectively. Then, 
\begin{itemize}
\item[a)] the inequality
            $$\int_0^T{\left(\gamma(t)f(t,N^{*,0},0)-\delta_1(t)\right)dt}<0$$
            implies the extinction of predators  for \eqref{PP}, i.e. $P(t)\to 0$ for $t\to + \infty$. Moreover we have $|N(t)-N^{*,0}(t)|\to 0$ for $t\to + \infty$.
\item[b)] the inequality
            $$\int_0^T{\left(\gamma(t)f(t,N^{*,0},0)-\delta_1(t)\right)dt}>0$$
            implies the uniform persistence of both predators and prey for \eqref{PP}.
\end{itemize}
\end{tr}

\begin{proof}
Thanks to the previous result, the proof in \cite[Proof of Theorem 3.2]{garrione2016} holds also in this case without modifications. As invariant and absorbing set for system \eqref{PP} we consider for a small $\varepsilon>0$ 
\begin{equation}\label{Omega}
{\cal K}'=\left\{(N,P)\in \mathbb{R}^2_{+}: 0\leq N\leq \overline K_++\varepsilon,\; N+\dfrac{P}{\bar{\gamma}}\leq (\overline{K}_++\varepsilon)\left(1+\dfrac{\overline r}{\underline \delta_1 }\right) \right \},
\end{equation}
with $\overline{r}:=\max_{t\in[0,T]}{k(t,\xi(t))}$. In fact, the axis are trajectories and we have that the vector field  $F(t,N,P)=(F_1(t,N,P),F_2(t,N,P))$ associated to \eqref{PP} satisfies:
$$F(t,N, P)_{|N=\overline K_++\varepsilon}\cdot (1,0)=F_1(t,\overline K_++\varepsilon, P)=k(t,\overline K_++\varepsilon)(\overline K_++\varepsilon) -f(t,\overline K_++\varepsilon,P)P< 0,$$ 
if $P>0,$ and, being for each $c\ge 0$ 
$$
\dfrac{\hat{P}_c}{\bar{\gamma}}=(\overline{K}_++\varepsilon)\left(1+\dfrac{\overline r}{\underline \delta_1 }\right)-N+c,$$
\begin{displaymath}
\begin{array}{lll}
F(t,N,\hat{P}_c)\cdot \left(1,\dfrac{1}{\overline{\gamma}}\right)&\le& 
\left( k(t,N)N -\dfrac{\delta_1}{\underline{\delta}_1}\overline r (\overline K_++\varepsilon)\right)- \delta_1 ((\overline K_++\varepsilon) -N+c)\\
&
-& \hat{P}_cf(t,N,\hat{P}_c)\left( 1-\dfrac{\gamma}{\bar{\gamma}} \right)
< 0.\end{array}
\end{displaymath}
As in \cite{garrione2016} we obtain the existence of $\eta>0$ such that $\liminf_{t\to +\infty} P(t)>\eta$ if $N(0)~>~0,$ $P(0)>0$. Now as in the mentioned paper easily follows the existence of $\delta>0$ such that  $\limsup_{t\to +\infty} N(t)>\delta$ if $N(0)>0,\; P(0)>0$ and by \cite[Theorem 1.3.3]{zhao2003}, we have that the system is uniformly persistent. 
\end{proof}

\begin{tr}
Under the assumptions of Theorem 2, part (b), there exists a non-trivial $T$-periodic solution $(N(t),P(t))$ to \eqref{PP} such that $N(t),\; P(t)>0$ for every $t\in[0,T]$.
\end{tr}

\begin{proof}
This result is a consequence of  \cite[Theorem 1.3.6]{zhao2003}, applied to the Poincar\'e map on the invariant set.
\end{proof}

\rem{The basic reproduction number $R_0$ \cite{georgescu07, garrione2016} of the system is the number of predators one predator gives rise during its life. In this case, when there are no predators, the prey population approximates the periodic solution  $N^{*,0}$. Hence we have  
$$R_0=\frac{\int_0^T\gamma(t)f(t,N^{*,0},0)dt}{\int_0^T\delta_1(t)dt}$$
and  Theorem \ref{teo_oer_weak} states that $R_0<1$ implies the extinction of the predators, while $R_0>1$ its persistence. In case of extinction we conclude that the prey population converges to the non-trivial periodic orbit in the predator-free space. In the following section we will see that the case of strong Allee effect is more delicate as there are two non-trivial periodic orbits in this space and the origin is stable.}

\begin{rem}It is worthwhile to note that the conditions obtained in Theorem \ref{teo_oer_weak} also holds in the autonomous case. As far as we know, in literature the autonomous predator--prey model with a weak Allee effect on the prey growth and a functional response satisfying hypothesis (f1)--(f5) is not studied (note that in \cite{sen2012bifurcation} the authors consider a singular ratio-dependent trophic singular function, which is excluded in the present study), and then the comparison of the obtained results is not possible.
\end{rem}


\section{Existence of periodic orbits with a strong Allee effect}\label{Sec:Strong}
Assuming now that the prey growth function describes a strong Allee effect (properties (gs1)--(gs4)), we can prove the following result  in absence of predators. This result generalizes an analogous one in \cite{rizaner12}.

\begin{tr}\label{th:sA0p}
Assume that the growth function satisfies the listed properties for a strong Allee effect (gs1)--(gs4). The dynamics is described by 
\begin{equation}\label{strong_preyonly}
\dot{N}=k(t,N)N,\qquad N\ge 0.
\end{equation} 
Then, equation \eqref{strong_preyonly}  admits
\begin{itemize}
\item the trivial solution (total extinction), which is locally asymptotically stable,
\item a positive, bounded and unstable $T$-periodic solution  $N^{-}(t)$,
\item a positive, bounded, locally asymptotically stable $T$-periodic solution $N^{+}(t)$. 
\end{itemize}
\end{tr}

\begin{proof}
Our proof follows \cite[p.126--129]{hale2012}.

The trivial solution $N(t)\equiv 0$ is a $T$-periodic solution. Let us consider a solution with positive initial data $N(0)>0$.

Then if $t$ is such that $N(t)>\overline{K}_+$ or $0<N(t)<\underline{K}_-$, we have $\dot{N}(t)<0$; analogously, if  $\overline{K}_-<N(t)<\underline{K}_+$, then $\dot{N}(t)>0$. Therefore, any solution $N(t)$ with initial data $N(0)\ge 0$ must be bounded in the future.  Depending on the initial condition, one of the following cases holds:
\begin{itemize}
\item the solution $N(t)$ approaches zero as $t \to +\infty$ and approaches a $T$-periodic solution $N^{*}_-(t)$ as $t \to -\infty$,
\item the solution $N(t)$ approaches a $T$-periodic solution $N^{*}_-(t)$ as $t \to -\infty$ (we use the same notation as in the previous case and will see below that in fact there is only one solution in these conditions) and approaches a $T$-periodic solution $N^{*}_{+}(t)$ as $t \to +\infty$,
\item the solution $N(t)$ approaches $+\infty$ as $t \to -\infty$ and approaches a $T$-periodic solution $N^{*}_{+}(t)$ as $t \to +\infty$.
\end{itemize}
If $N(0)<\underline{K}_-$ the first one occurs, while if $\overline{K}_-< N(0)< \underline{K}_+ $ the second one holds. Finally if $N(0)>\overline{K}_+$ we are in the third case. 
Furthermore, these $T$-periodic solutions have the properties 
$$\underline{K}_-\leq N^{*}_-(t)\leq \overline{K}_-,\qquad \underline{K}_+\leq N^{*}_+(t)\leq \overline{K}_+,$$ 
and therefore they are positive.

We want now to prove the uniqueness of these $T$-periodic solutions in their existence intervals. Regarding $N^{*}_+(t)$ in $[ \underline{K}_+, \overline{K}_+]$, the same proof in the case of a weak Allee effect holds. We have now to prove that $N^{*}_-(t)$ is the unique $T$-periodic solution in $[\underline{K}_-, \overline{K}_-]$. Assuming that there are two such solutions $N^{*,-}_1(t)>N^{*,-}_2(t)$, we consider $w^{*,-}_i(t)=N^{*,-}_i(-t)$ two solutions of the time-reverse equation 
\begin{equation}\label{strong_preyonly_reverse}
\dot{w}(t)=-k(-t,w(t))w(t),
\end{equation} 
such that $v(t):=w^{*,-}_1(t) - w^{*,-}_2(t)>0$ for all $t\in[0,T]$.


Then we have
\begin{align*}
\dot{v}(t)&=\dot{w}^{*,-}_1(t)-\dot{w}^{*,-}_2(t)\\
              &=-k(-t,w^{*,-}_1(t)) w^{*,-}_1(t) +k(-t,w^{*,-}_2(t)) w^{*,-}_2(t)\\
              &=-k(-t,w^{*,-}_1(t)) \left(w^{*,-}_1(t)-w^{*,-}_2(t)\right)+ \left( k(-t,w^{*,-}_2(t)) -k(-t,w^{*,-}_1(t))\right)w^{*,-}_2(t).
\end{align*}
Since $k(-t,w)$ is increasing with respect to $w$ in $[\underline{K}_-, \overline{K}_-]$, we have that $ k(t,w^{*,-}_2(t)) -k(t,w^{*,-}_1(t))<0$, and then the same arguments of the previous case hold.
\end{proof}

We want now to obtain results for the predator--prey system. Let us begin by analyzing the behavior near the $(N^*_{\pm}(t),0)$.

\begin{tr}
Assume that the prey growth function and the functional response satisfy the listed properties for a strong Allee effect (gs1)--(gs4) and (f1)--(f5), respectively.  Then, the origin is always locally asymptotically stable. Moreover 
$$\lambda^\pm_1=\exp{\int_0^T \left({\frac{\partial k}{\partial N}(t,N^*_{\pm}(t))N^*_{\pm}(t)}\right)dt} \quad \mbox{ and }\quad\lambda^\pm_2= \exp{\int_0^T {\left(\gamma(t)f(t,N^*_{\pm}(t),0)-\delta_1(t)\right)}dt}$$
 are the eigenvalues of the monodromy matrix associated to the linearized system on the periodic solutions $(N^*_{\pm}(t),0).$ We have $\lambda^-_1>1>\lambda^+_1$ and hence $(N^*_-(t),0)$ is unstable while the stability of $(N^*_+(t),0)$ depends on the sign of  $\int_0^T\left(\gamma(t)f(t,N^*_+(t),0)-\delta_1(t)\right)dt.$
\end{tr}


\begin{proof}
The linearized equation of system \eqref{PP} at $(0,0)$ is
\begin{align*}
\dot{u}&=k(t,0) u,\\
\dot{v}&=-\delta_1(t)v,
\end{align*}
then $(0,0)$ is always locally asymptotically stable.

The linearization of \eqref{PP} at $(N^*_{\pm}(t),0)$ is
\begin{align*}
\dot{u}&=\left(\frac{\partial k}{\partial N}(t,N^*_{\pm}(t))N^*_{\pm}(t)+k(t,N^*_{\pm}(t)) \right)u- f(t,N^*_{\pm}(t),0) v,\\
\dot{v}&=\left(\gamma(t)f(t,N^*_{\pm}(t),0)-\delta_1(t)\right) v.
\end{align*}
Let $U(t)$ be the canonical fundamental matrix, and set $t=T$. Then
$$
U(T)=\begin{pmatrix}
\displaystyle{\exp{\int_0^T \left( \frac{\partial k}{\partial N}(t,N^*_{\pm}(t))N^*_{\pm}(t)+k(t,N^*_{\pm}(t))\right)dt}} & *\\[0.2cm]
0 & \displaystyle{\exp{\int_0^T {\left(\gamma(t)f(t,N^*_{\pm}(t),0)-\delta_1(t)\right)}dt}}\\
\end{pmatrix},$$
where $*$ denotes some constant depending on $T$.

Now, using the fact that 
$$k(t,N^*_{\pm}(t))=\frac{\dot N^*_{\pm}(t)}{N^*_{\pm}(t)}$$
 and recalling that $N^*_{\pm}(t)$ are $T$-periodic, we conclude that the two eigenvalues of $U(T)$ are 
$$\lambda^\pm_1=\exp{\int_0^T {\left(\frac{\partial k}{\partial N}(t,N^*_{\pm}(t))N^*_{\pm}(t) \right)dt}}\quad \textnormal{and}\quad \lambda^\pm_2=\displaystyle{\exp{\int_0^T {\left(\gamma(t)f(t,N^*_{\pm}(t),0)-\delta_1(t)\right)}dt}}.$$ 
We easily conclude now that   $\lambda^-_1>1>\lambda^+_1$.
Then $(N^*_-(t),0)$ is always unstable, while the stability of $(N^*_+(t),0)$  is related to the eigenvalue $\lambda_2^+$.  
\end{proof}

\begin{rem}
With a strong Allee effect, when there are no predators, the prey population approximates either zero or the periodic solution $N^*_+(t)$. We cannot expect persistence in the first quadrant and obtain the existence of a periodic orbit as a consequence of persistence (as in the case of weak Allee effect).  Thus we will  use other techniques.
\end{rem}

We can also note that 
$$\int_0^T {\left(\gamma(t)f(t,N^*_-(t),0)-\delta_1(t)\right)}dt < \int_0^T {\left(\gamma(t)f(t,N^*_+(t),0)-\delta_1(t)\right)}dt,$$
then the following cases are possible:
\begin{itemize}
\item[a)] $\lambda^-_2<\lambda^+_2<1$,
\item[b)] $\lambda^-_2<1<\lambda^+_2$,
\item[c)] $1<\lambda^-_2<\lambda^+_2$.
\end{itemize}

We follow \cite[Thm.3]{zhidong1999}. Define the Poincar\'e map $P_T:\mathbb{R}_+^2\to \mathbb{R}_+^2$ as follows: for any $X_0\in \mathbb{R}_+^2$, let $X(t,X_0)$ be the solution of system \eqref{PP} with initial condition $X(0,X_0)=X_0$, then $P_T(X_0)=X(T,X_0)$. We have that $P_T$ is a continuous map on $\mathbb{R}_+^2$. Furthermore, as in the proof of Theorem \ref{teo_oer_weak}, for a small $\varepsilon>0$ the set ${\cal K}'$ defined in \eqref{Omega}, 
is a positive invariant set for system \eqref{PP} in case of a strong Allee effect. By reflecting the picture with respect to the $N$ and $P$ axes, we extend the map $P_T$ from $Q$ to $S:=\{(N,P)\in \mathbb{R}^2: (|N|,|P|)\in {\cal K}'\}$ which is a convex set. Note that the positive $N$ and $P$ axes are invariant sets of system \eqref{PP}. We have that $P_T$ is continuous on $S$, $P_T(S)\subset S$ and  there are no fixed points of $P_T$ in the boundary of $S$. By the Leray--Schauder principle \cite{Zeidler1986} we have 
$$deg(I-P_T,S,(0,0))=1,$$
where $deg$ is the Brouwer degree in the plane \cite{Zeidler1986} and $I$ is the identical map. We can also note that in $S$ we have at least $5$ fixed points of $P_T$:
\begin{itemize}
\item a fixed point corresponds to $(0,0)$,
\item two fixed points correspond to $(N^*_-(t),0)$ and $(-N^*_-(t),0)$,
\item two fixed points correspond to $(N^*_+(t),0)$ and $(-N^*_+(t),0)$,
\item other fixed points, corresponding to positive $T$-periodic positive solutions, could be present.
\end{itemize}

Since $(0,0)$ is always a stable node \cite{krasnosel1968}, the index of the fixed point \cite{Zeidler1986} $i(P_T,(0,0))=1$. Then we have the following cases:
\begin{itemize}
\item[a)] $\lambda^-_2<\lambda^+_2<1$: the indexes of $P_T$ at $(N^*_-(t),0)$ and $(N^*_+(t),0)$ are $-1,\;1$, respectively (Figure \ref{case1}). 
\item[b)] $\lambda^-_2<1<\lambda^+_2$: the indexes of $P_T$ at $(N^*_-(t),0)$ and $(N^*_+(t),0)$ are $-1,\;-1$, respectively (Figure \ref{case2}). 
\item[c)] $1<\lambda^-_2<\lambda^+_2$: the indexes of $P_T$ at $(N^*_-(t),0)$ and $(N^*_+(t),0)$ are $1,\;-1$, respectively (Figure \ref{case3}).  
\end{itemize}

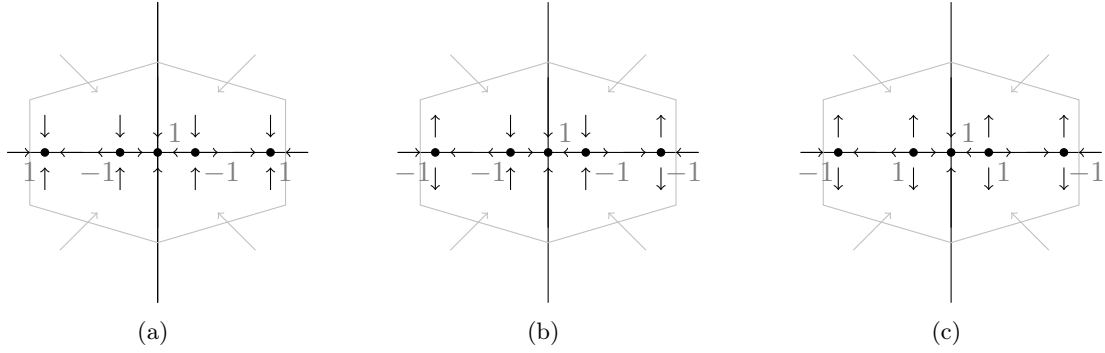
\begin{figure}[!ht]
\subfigure[\label{case1}]{
\begin{tikzpicture}[scale=1]
\draw (0,2) -- (0,-2);
\draw (-2,0) -- (2,0);

\draw (0,2) -- (0,-2);
\draw[gray!50!white] (-1.7,-0.7) -- (-1.7,0.7) -- (0,1.2) -- (1.7,0.7) -- (1.7,-0.7) -- (0,-1.2) -- (-1.7,-0.7);
\draw [->, gray!50!white] (1.3,1.3) -- (0.8,0.8);
\draw [->, gray!50!white] (-1.3,1.3) -- (-0.8,0.8);
\draw [->, gray!50!white] (1.3,-1.3) -- (0.8,-0.8);
\draw [->, gray!50!white] (-1.3,-1.3) -- (-0.8,-0.8);

\draw [->] (0,1) -- (0,0.2); 
\draw [->] (0.5,0) -- (0.2,0); 
\draw [->] (0.5,0.5) -- (0.5,0.2);
\draw [->] (0.5,-0.5) -- (0.5,-0.2);
\draw [->] (0.5,0) -- (0.8,0); 
\draw [->] (2,0) -- (1.7,0);
\draw [->] (1,0) -- (1.3,0);
\draw [->] (1.5,0.5) -- (1.5,0.2);
\draw [->] (1.5,-0.5) -- (1.5,-0.2);

\draw [->] (0,-1) -- (0,-0.2);%
\draw [->] (-0.5,0) -- (-0.2,0);
\draw [->] (-0.5,0.5) -- (-0.5,0.2);%
\draw [->] (-0.5,-0.5) -- (-0.5,-0.2);%
\draw [->] (-0.5,0) -- (-0.8,0);
\draw [->] (-2,0) -- (-1.7,0);
\draw [->] (-1,0) -- (-1.3,0);
\draw [->] (-1.5,0.5) -- (-1.5,0.2);%
\draw [->] (-1.5,-0.5) -- (-1.5,-0.2);%

\draw [fill] (0,0) circle [radius=0.05];
\draw [fill] (0.5,0) circle [radius=0.05];
\draw [fill] (1.5,0) circle [radius=0.05];
\draw [fill] (-0.5,0) circle [radius=0.05];
\draw [fill] (-1.5,0) circle [radius=0.05];

\node [above right] at (0,0) {\color{gray}$1$};
\node [below] at (-1.7,0) {\color{gray}$1$};
\node [below] at (-0.8,0) {\color{gray}$-1$};
\node [below] at (1.7,0) {\color{gray}$1$};
\node [below] at (0.85,0) {\color{gray}$-1$};
\end{tikzpicture}
}
\hspace{0.5cm}
\subfigure[\label{case2}]{
\begin{tikzpicture}[scale=1]
\draw (0,2) -- (0,-2);
\draw (-2,0) -- (2,0);

\draw[gray!50!white] (-1.7,-0.7) -- (-1.7,0.7) -- (0,1.2) -- (1.7,0.7) -- (1.7,-0.7) -- (0,-1.2) -- (-1.7,-0.7);
\draw [->, gray!50!white] (1.3,1.3) -- (0.8,0.8);
\draw [->, gray!50!white] (-1.3,1.3) -- (-0.8,0.8);
\draw [->, gray!50!white] (1.3,-1.3) -- (0.8,-0.8);
\draw [->, gray!50!white] (-1.3,-1.3) -- (-0.8,-0.8);

\draw [->] (0,1) -- (0,0.2);
\draw [->] (0.5,0) -- (0.2,0);
\draw [->] (0.5,0.5) -- (0.5,0.2);
\draw [->] (-0.5,0.5) -- (-0.5,0.2);
\draw [->] (0.5,0) -- (0.8,0);
\draw [->] (2,0) -- (1.7,0);
\draw [->] (1,0) -- (1.3,0);
\draw [->] (1.5,0.2) -- (1.5,0.5);
\draw [->] (-1.5,0.2) -- (-1.5,0.5);

\draw [->] (0,-1) -- (0,-0.2);
\draw [->] (-0.5,0) -- (-0.2,0);
\draw [->] (0.5,-0.5) -- (0.5,-0.2);
\draw [->] (-0.5,-0.5) -- (-0.5,-0.2);
\draw [->] (-0.5,0) -- (-0.8,0);
\draw [->] (-2,0) -- (-1.7,0);
\draw [->] (-1,0) -- (-1.3,0);
\draw [->] (1.5,-0.2) -- (1.5,-0.5);
\draw [->] (-1.5,-0.2) -- (-1.5,-0.5);

\draw [fill] (0,0) circle [radius=0.05];
\draw [fill] (0.5,0) circle [radius=0.05];
\draw [fill] (1.5,0) circle [radius=0.05];
\draw [fill] (-0.5,0) circle [radius=0.05];
\draw [fill] (-1.5,0) circle [radius=0.05];

\node [above right] at (0,0) {\color{gray}$1$};
\node [below] at (-1.8,0) {\color{gray}$-1$};
\node [below] at (-0.8,0) {\color{gray}$-1$};
\node [below] at (1.8,0) {\color{gray}$-1$};
\node [below] at (0.85,0) {\color{gray}$-1$};
\end{tikzpicture}       
}    
\hspace{0.5cm}
\subfigure[\label{case3}]{
\begin{tikzpicture}[scale=1]
\draw (0,2) -- (0,-2);
\draw (-2,0) -- (2,0);

\draw[gray!50!white] (-1.7,-0.7) -- (-1.7,0.7) -- (0,1.2) -- (1.7,0.7) -- (1.7,-0.7) -- (0,-1.2) -- (-1.7,-0.7);
\draw [->, gray!50!white] (1.3,1.3) -- (0.8,0.8);
\draw [->, gray!50!white] (-1.3,1.3) -- (-0.8,0.8);
\draw [->, gray!50!white] (1.3,-1.3) -- (0.8,-0.8);
\draw [->, gray!50!white] (-1.3,-1.3) -- (-0.8,-0.8);

\draw [->] (0,1) -- (0,0.2);
\draw [->] (0.5,0) -- (0.2,0);
\draw [->] (0.5,0.2) -- (0.5,0.5);
\draw [->] (0.5,-0.2) -- (0.5,-0.5);
\draw [->] (0.5,0) -- (0.8,0);
\draw [->] (2,0) -- (1.7,0);
\draw [->] (1,0) -- (1.3,0);
\draw [->] (1.5,0.2) -- (1.5,0.5);
\draw [->] (1.5,-0.2) -- (1.5,-0.5);

\draw [->] (0,-1) -- (0,-0.2);
\draw [->] (-0.5,0) -- (-0.2,0);
\draw [->] (-0.5,0.2) -- (-0.5,0.5);
\draw [->] (-0.5,-0.2) -- (-0.5,-0.5);
\draw [->] (-0.5,0) -- (-0.8,0);
\draw [->] (-2,0) -- (-1.7,0);
\draw [->] (-1,0) -- (-1.3,0);
\draw [->] (-1.5,0.2) -- (-1.5,0.5);
\draw [->] (-1.5,-0.2) -- (-1.5,-0.5);

\draw [fill] (0,0) circle [radius=0.05];
\draw [fill] (0.5,0) circle [radius=0.05];
\draw [fill] (1.5,0) circle [radius=0.05];
\draw [fill] (-0.5,0) circle [radius=0.05];
\draw [fill] (-1.5,0) circle [radius=0.05];
\node [above right] at (0,0) {\color{gray}$1$};
\node [below] at (-1.8,0) {\color{gray}$-1$};
\node [below] at (-0.7,0) {\color{gray}$1$};
\node [below] at (1.8,0) {\color{gray}$-1$};
\node [below] at (0.7,0) {\color{gray}$1$};
\end{tikzpicture}       
}
\caption{Graphical representation, where $1$ and $-1$ are the the fixed point indexes associated to the corresponding periodic solutions.}\label{threeCases}
\end{figure}

At first we state a result about extinction
\begin{tr}
Assume that the prey growth function and the functional response satisfy the listed properties (strong Allee effect).  If  $\lambda^-_2<\lambda^+_2<1$, then there is extinction of the predators.
\end{tr}

\begin{proof}
We consider a solution $(N,P)$ of \eqref{PP}. From the first equation we have that 
\begin{equation}
N'\le k(t,N)N.
\label{ineq1}
\end{equation}
Let us consider the solution of  $N_1'=k(t,N_1)N_1$ such that $N_1(0)=N(0)$. For each $\varepsilon>0$  there exists a $t_\varepsilon$ such that $N_1(t)\le N^*_+(t)+\varepsilon$ for each $t\ge t_\varepsilon$. Using the monotonicity of $f$ with respect to $N$ we have that for each $t\ge t_\varepsilon$ 
\begin{equation}\label{ineqPP}
\begin{cases}
N(t) \le N_1(t),\\
\dot{P}(t)\le\gamma(t)f(t,N^*_+(t)+\varepsilon,0)P(t)-\delta_1(t)P(t).
\end{cases}
\end{equation}
Since $\lambda^+_2<1$ and choosing $\varepsilon$ sufficiently small, we have that $\lim_{t\to +\infty} P(t)=~0$.
\end{proof}

\begin{rem}
In this case two different outcomes are possible when $t\to+\infty$, depending on the initial conditions: we can have the total extinction or the $N$-component of the solution tends to the $T$-periodic orbit $N_+^*(t)$.
\end{rem}

\noindent 
A first result on existence of positive $T$-periodic orbits is the following.

\begin{tr}
Assume that the prey growth function and the functional response satisfy the listed properties (strong Allee effect).  If $\lambda^-_2<1<\lambda^+_2$ then there exists at least a positive $T$- periodic orbit $(N^*,P^*)$. 
\end{tr}

\begin{proof} 
This result is an immediate consequence of the indexes of the fixed points of $P_T$  on the $N$-axis $(\pm N^*_{\pm}(0),0)$ and $(0,0)$  taking into account that $deg(I-P_T,S,(0,0))=1.$ In fact either there are infinite $T$-periodic orbits or there is at least one positive periodic orbit with a positive index (four in total, one in each quadrant).
\end{proof}

In  Section \ref{Sec:NumRes} we present some simulations which illustrate the existence of this periodic orbit. In the simulations it seems that the orbit can be stable or unstable, depending on a parameter.

From \cite[Theorem 2 in Section 2.5]{ortega1995} the degree of a Lyapunov stable and isolated periodic orbit is one as the Poincar\'e map is orientation preserving. 
It seems from our simulations that in case
$1<\lambda^-_2<\lambda^+_2$ there are no positive periodic orbits. But as a consequence of this remark, if this is not true and there is not an in an infinite number positive periodic orbits, if there is one stable then there is one unstable.

\vspace{3mm}
We state now a lemma in which we give a condition for the instability of positive $T$-periodic solutions. The case of stability would conduct us to a much more complex condition and hence we do not state it. 

\begin{lm}
Let $(N^*(t), P^*(t))$ be a positive $T$-periodic solution of \eqref{PP}and set 
{\small $$\alpha=\int_0^T \left(\frac{\partial k}{\partial N}(t,N^*)N^*-\frac{\partial f}{\partial N}(t,N^*,P^*)P^*+f(t,N^*,P^*)\frac{P^*}{N^*}+ \gamma(t)\frac{\partial f}{\partial P}(t,N^*,P^*){P^*}-\delta_2(t)P^*\right) dt.$$}
If $\alpha>0$ the periodic solution is unstable.
\label{alpha}
\end{lm}

\begin{proof}
We consider the following change of variables 
$$u=\frac{N}{N^*}-1\,\,\,v=\frac{P}{P^*}-1,$$
which transforms the solution $(N^*(t), P^*(t))$ to the origin. Then we linearize the system  at $(0,0)$ and we get 
\begin{align*}
\dot{u}&=\left(\frac{\partial k}{\partial N}(t,N^*)N^*-\frac{\partial f}{\partial N}(t,N^*,P^*)P^*+f(t,N^*,P^*)\frac{P^*}{N^*}\right) u-\\
&- \left(\frac{\partial f}{\partial P}(t,N^*,P^*)\frac{{P^*}^2}{N^*}+f(t,N^*,P^*)\frac{P^*}{N^*} \right)v,\\
\\
\dot{v}&=\left(\gamma(t)\frac{\partial f}{\partial N}(t,N^*,P^*)N^*\right)u+\left(\gamma(t)\frac{\partial f}{\partial P}(t,N^*,P^*)P^*-\delta_2 P^*\right) v,
\end{align*}
where $N^*$ and $P^*$ are evaluated at $t$.
The result is an immediate consequence of Liouville's formula, see also \cite[II, Chapter 7, 3]{YS}
\end{proof}

This result can be specialized in the case of a Gilpin prey-growth function and a Holding-type II functional response (widely used in literature), obtaining the following result.
\begin{cor}\label{alphaH}
Assume that 
$$k(t,N)=r(t)(N-K_-(t))(K_+(t)-N) \qquad and \qquad f(t,N,P)=\dfrac{b(t)p(t)N}{1+p(t)N}.$$
Let $(N^*(t), P^*(t))$ be a positive $T$-periodic solution of \eqref{PP}.
If
$$\alpha:=\int_0^T \left(r(t)(K_+(t)+K_-(t)-2N^*(t))N^*(t)+\frac{b(t)p^2(t)N^*(t)P^*(t)}{(1+p(t)N^*(t))^2}-\delta_2(t)P^* \right)dt>0$$
then $(N^*(t), P^*(t))$ is unstable. 
\end{cor}

\medskip Note that if $\delta_2\equiv 0$ and 
$$\int_0^T \left(r(t)(K_+(t)+K_-(t)-2N^*(t))N^*(t)\right) dt \geq 0,$$
 then $(N^*(t), P^*(t))$ is unstable.

Also, by \cite[Theorem 7.1]{KPPZ} we have that if $(P^T)'\neq I$, where we denote by $I$ the $2\times 2$ identity matrix, then the fixed point index of the fixed points of $P^T$ can be either $1$, $0$ or $-1$. Thus, taking into account Liouville's formula,  if the value $\alpha$ defined in Lemma \ref{alpha} is nonzero in the corresponding solution
we have that $(P^T)'\neq I$ and  the 
fixed point index  can take one of the three mentioned  values. 

An immediate consequence of this remark is the following 

\begin{tr}
Assume that the prey growth function and the functional response satisfy the listed properties (strong Allee effect) and that $\lambda^-_2<1<\lambda^+_2$.
If there is not an infinite number of periodic solutions and for each periodic solution the corresponding $\alpha\neq 0$ then the number of stable periodic orbits is smaller or equal to the number of unstable ones plus one. 
\end{tr}

\section{Leslie--Gower}\label{Sec:LG}
We consider in this section a modified Leslie--Gower predator--prey model, in which, as in the previous sections, the prey growth and the functional response satisfy the listed properties for a strong (weak) Allee effect (gs1)--(gs4) ((gw1)--(gw4)) and (f1)--(f5), respectively. The system writes

\begin{equation}\label{mLG}
\begin{cases}
\dot{N}=k(t,N)N-f(t,N,P)P,\\
\dot{P}=c_2(t)P\left(1-\dfrac{P}{N+c(t)} \right),
\end{cases}
\end{equation}
where $N$ is the prey and $P$ the predator, the $T-$periodic functions $k$ and $f$ are the prey growth and the functional response, and $c,\;c_2$ are positive, $T$-periodic functions. 

The case of weak Allee effect can be treated, taking into account Theorem \ref{th:wA0p}.
We now analyze the case of a strong Allee effect. In absence of predators, the same arguments of Theorem \ref{th:sA0p} hold: the system admits two $T$-periodic solutions, $N_{\pm}^*(t)$. It can be also shown that, when the prey is absent, the system admits a $T$-periodic solution $P_0^*(t)$, and in particular that $$\underline{c}\leq P_0^*(t) \leq \overline{c}.$$ 

The linearization of \eqref{mLG} at $(0,0)$ is
$$\begin{pmatrix} \dot{u}\\\dot{v}\end{pmatrix}=
\begin{pmatrix} k(t,0)&0\\0&c_2(t)\end{pmatrix}
\begin{pmatrix} u\\v\end{pmatrix},$$
then,  $(0,0)$ is always unstable (saddle point). 

The linearization of \eqref{mLG} at $(N^*_{\pm},0)$ is
$$\begin{pmatrix} \dot{u}\\\dot{v}\end{pmatrix}=
\begin{pmatrix} k(t,N^*_{\pm})+\dfrac{\partial k}{\partial N}(t,N^*_{\pm})N^*_{\pm}&
*\\0&c_2(t)\end{pmatrix}
\begin{pmatrix} u\\v\end{pmatrix},$$
and the canonical fundamental matrix $U(t)$ with $t=T$ is
$$U(T)=
\begin{pmatrix} \exp{\int_0^T \left(k(t,N^*_{\pm}(t))+\dfrac{\partial k}{\partial N}(t,N^*_{\pm}(t))N^*_{\pm}(t))dt\right) } &*\\
0&\exp{\int_0^T c_2(t)}dt\end{pmatrix}.$$
As in the previous section, we conclude that the two eigenvalues of $U(T)$ are
$$\lambda_1^{\pm}=\exp{\int_0^T \dfrac{\partial k}{\partial N}(t,N^*_{\pm}(t))N^*_{\pm}(t)}dt
\qquad \textnormal{and} \qquad \lambda_2^{\pm}=\exp{\int_0^T c_2(t)}dt.$$
We have that $\lambda_1^->1>\lambda_1^+$ and $\lambda_2^-=\lambda_2^+>1$, then both $N^*_{\pm}$ are unstable. 

Finally, the linearization of \eqref{mLG} at $(0,P_0^*)$ is
$$\begin{pmatrix} \dot{u}\\\dot{v}\end{pmatrix}=
\begin{pmatrix} k(t,0)-\dfrac{\partial f}{\partial N}(t,0,P_0^*(t))P_0^*(t)&-\dfrac{\partial f}{\partial P}(t,0,P_0^*(t))P_0^*(t)\\[0.3cm]
\dfrac{c_2(t)P_0^{*2}(t)}{c(t)^2}&c_2(t)\left(1-\dfrac{2P_0^*(t)}{c(t)}\right)\end{pmatrix}
\begin{pmatrix} u\\v\end{pmatrix}.$$
Note that, if the functional response is only prey-dependent we have
\begin{equation}\dfrac{\partial f}{\partial P}(t,0,P_0^*(t))=0;
\label{eq:preydependent}
\end{equation}
it is worthwhile to note that this also holds for some predator-dependent functional responses (for instance the Beddington--DeAngelis and the Watt, see \cite{garrione2016}), and it seems biologically reasonable.
If we assume \eqref{eq:preydependent} and taking into account that
$$\int_0^T c_2(t)\left(1-\dfrac{P_0^*}{c(t)} \right)dt=0,$$
then the canonical fundamental matrix $U(t)$ with $t=T$ writes
$$U(T)=
\begin{pmatrix} \exp{\displaystyle\int_0^T \left(k(t,0)-\dfrac{\partial f}{\partial N}(t,0,P_0^*(t))P_0^*(t)dt \right)}&0\\[0.3cm]
*&\exp{\left(-\displaystyle\int_0^T c_2(t)\dfrac{P_0^*}{c(t)}dt\right)}\end{pmatrix}.$$
The two eigenvalues of $U(T)$ are
$$\lambda_1=\exp{\int_0^T \left(k(t,0)-\dfrac{\partial f}{\partial N}(t,0,P_0^*(t))P_0^*(t)\right)dt}
 \textnormal{ and } 
\lambda_2=\exp{\left(-\int_0^T c_2(t)\dfrac{P_0^*(t)}{c(t)}dt\right)}<1.$$
With a strong Allee effect on the prey and a prey-dependent trophic function we have that $k(t,0)<0$, and $(0,P_0^*)$ turns out to be stable, while with a weak Allee effect or a predator-dependent trophic function, its stability depends on the parameter values. Then, thanks to the sign of $k(t,0)$, the strong Allee effect stabilizes this $T$-periodic solution, with respect to when a weak Allee effect or a logistic growth \cite{song2008dynamic} are considered.We will assume this case in the following.

The indexes in the expanded domain are illustrated in Figure \ref{LGdegree}; since we have that the sum of the indexes is equal to one, the system can exhibit one of the following cases (when a strong Allee effect is considered):
\begin{itemize}
\item there are no other non-trivial $T$-periodic orbits;
\item there are infinite non-trivial $T$-periodic orbits;
\item if there is a stable $T$-periodic orbit  then there is also an unstable.
\end{itemize}
As in the previous section, we have the following result, where now
$$\alpha=\int_0^T \left(\frac{\partial k}{\partial N}(t,N^*)N^*-\frac{\partial f}{\partial N}(t,N^*,P^*)P^*+f(t,N^*,P^*)\frac{P^*}{N^*}- \frac{2c_2(t)P^*}{N^*+c(t)}\right) dt.$$

\begin{tr}
Assume that the prey growth function and the functional response in system \eqref{mLG} satisfy the listed properties (strong Allee effect). If there is not an infinite number of periodic solutions and for each periodic solution the corresponding $\alpha\neq 0$, then the number of stable periodic orbits is smaller or equal to the number of unstable ones. 
\end{tr}

In the following section we will present numerical results showing that the the existence of a stable and of an unstable periodic orbits seems to be the common behavior of this class of models.

\begin{figure}[!ht]
\centering
\begin{tikzpicture}[scale=1.5]
\draw (0,2) -- (0,-2);
\draw (-2,0) -- (2,0);

\draw[gray!50!white] (-1.7,-0.7) -- (-1.7,0.7) -- (0,1.2) -- (1.7,0.7) -- (1.7,-0.7) -- (0,-1.2) -- (-1.7,-0.7);
\draw [->, gray!50!white] (1.3,1.3) -- (0.8,0.8);
\draw [->, gray!50!white] (-1.3,1.3) -- (-0.8,0.8);
\draw [->, gray!50!white] (1.3,-1.3) -- (0.8,-0.8);
\draw [->, gray!50!white] (-1.3,-1.3) -- (-0.8,-0.8);

\draw [->] (0,0.1) -- (0,0.4);
\draw [->] (0,1) -- (0,0.9);

\draw [->] (0.5,0) -- (0.2,0);
\draw [->] (0.5,0.2) -- (0.5,0.5);
\draw [->] (0.5,-0.2) -- (0.5,-0.5);
\draw [->] (0.5,0) -- (0.8,0);
\draw [->] (2,0) -- (1.7,0);
\draw [->] (1,0) -- (1.3,0);
\draw [->] (1.5,0.2) -- (1.5,0.5);
\draw [->] (1.5,-0.2) -- (1.5,-0.5);

\draw [->] (0,-0.1) -- (0,-0.4);
\draw [->] (0,-1) -- (0,-0.9);
\draw [->] (-0.5,0) -- (-0.2,0);
\draw [->] (-0.5,0.2) -- (-0.5,0.5);
\draw [->] (-0.5,-0.2) -- (-0.5,-0.5);
\draw [->] (-0.5,0) -- (-0.8,0);
\draw [->] (-2,0) -- (-1.7,0);
\draw [->] (-1,0) -- (-1.3,0);
\draw [->] (-1.5,0.2) -- (-1.5,0.5);
\draw [->] (-1.5,-0.2) -- (-1.5,-0.5);

\draw [->] (0.3,0.7) -- (0.1,0.7);
\draw [->] (-0.3,0.7) -- (-0.1,0.7);
\draw [->] (0.3,-0.7) -- (0.1,-0.7);
\draw [->] (-0.3,-0.7) -- (-0.1,-0.7);

\draw [fill] (0,0) circle [radius=0.05];
\draw [fill] (0.5,0) circle [radius=0.05];
\draw [fill] (1.5,0) circle [radius=0.05];
\draw [fill] (-0.5,0) circle [radius=0.05];
\draw [fill] (-1.5,0) circle [radius=0.05];
\draw [fill] (0,0.7) circle [radius=0.05];
\draw [fill] (0,-0.7) circle [radius=0.05];

\node [above right] at (-0.1,0) {\color{gray}$\scalebox{0.75}[1.0]{\( - \)}1$};
\node [below] at (-1.8,0) {\color{gray}$\scalebox{0.75}[1.0]{\( - \)}1$};
\node [below] at (-0.7,0) {\color{gray}$1$};
\node [below] at (1.8,0) {\color{gray}$\scalebox{0.75}[1.0]{\( - \)}1$};
\node [below] at (0.7,0) {\color{gray}$1$};
\node [below] at (0.2,1.2) {\color{gray}$1$};
\node [below] at (0.2,-0.7) {\color{gray}$1$};
\end{tikzpicture}       
\caption{Graphical representation, where $1$ and $-1$ are the the fixed point indexes associated to the corresponding periodic solutions. }\label{LGdegree}
\end{figure}
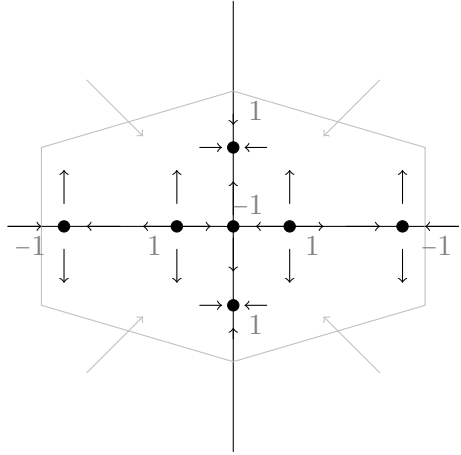

\section{Numerical results}\label{Sec:NumRes}
The aim of this section is to numerically investigate the features of systems \eqref{PP} and \eqref{mLG} in order to confirm the theoretical results and not to provide a complete picture of all the possible outcomes. 

\subsection{The Lotka--Volterra predator--prey model}
Here we take into account systems \eqref{PP}. We set:
$$k(t,N)=r(t)(N-K_-(t))(K_+(t)-N), \quad f(t,N)=\dfrac{b(t)p(t)N}{1+p(t)N},$$
following \cite{buffoni2011}. The periodic coefficients $r,\; K_+, \; \gamma, \; b, \; p,\; \delta_2$ are defined by the following expression for a generic parameter $c$ \cite{kuznetsov1992}: 
\begin{equation}\label{pc_fav}
c(t)=\hat{c}\left(1+s\sin\left(\dfrac{2\pi}{T} t\right)\right),
\end{equation}
where $\hat{c}$ represents the mean value on the period $T$ and $s\in [0,1]$ is a variability scaling factor (amplitude). This expression describes an increase in their values during the ``favorable'' season and a decrease during the ``unfavorable'' season. On the contrary, the periodic coefficient $\delta_1$, describing the mortality rate, and the prey extinction threshold $K_-$ are assumed to increase in the ``unfavorable'' season and to decrease in the ``favorable'' season. Then they are supposed to follow this general formulation
\begin{equation}\label{pc_unfav}
c(t)=\hat{c}\left(1+s\cos\left(\dfrac{2\pi}{T} t\right)\right).
\end{equation}
The parameters values used in the numerical simulations are listed in Table \ref{tab:parsim}, while we vary the parameter $\hat{p}$ in order to select different regimes of the systems. The results of the numerical simulations are shown in Figures \ref{a-t-estinzionePredatore}--\ref{c-t-estinzionetotale} for different values of $\hat{p}$.

\begin{table}[h!]
\centering
\begin{tabular}{ccccccccc}
\toprule
$\hat{r}$& $\hat{K}_-$& $\hat{K}_+$& $\hat{\gamma}$& $\hat{b}$& $\hat{\delta}_1$& $\hat{\delta}_2$& $s$& $T$\\
\midrule
0.11& 0.02& 1&0.39 & 0.88 &  0.19 & 0 & 0.1& 365\\
\midrule
$(d^{-1})$& & & & $(d^{-1})$&$(d^{-1})$&$(d^{-1})$& & $(d)$\\
\bottomrule
\end{tabular}
\caption{List of parameters values used in the numerical simulations and their unit of measure.}\label{tab:parsim}
\end{table}

In Figure \ref{a-t-estinzionePredatore}, obtained chosing $\hat{p}=1$, it can be seen that the solution starting from the initial condition $(0.2,0.1)$ tend to a $T$-periodic orbit where the predators are extinct. In Figure \ref{b-t-orbitaTperiodica}, starting with the same initial condition as before and with $\hat{p}=1.3$, the solution tends to a coexistence $T$-periodic orbit, shown in Figure \ref{b-r-orbitaTperiodica} in the phase-space $(N,P,t)$. Red dots indicate the Poincar\'e map of period $T$. Increasing the value of $\hat{p}$, we are able to find stable periodic orbits of different periods ($2T,\; 4T, \dots$). For $\hat{p}=1.522$, it is very difficult to recognize any periodicity of the solution. For $\hat{p}=1.55$ we obtain a stable periodic orbit of period $3T$ (Figures \ref{b-t-orbita3Tperiodica} and \ref{b-r-orbita3Tperiodica}). Finally, with $\hat{p}=1.8$ the only  outcome we get is the total extinction (Figure \ref{c-t-estinzionetotale}) and we are not able to find the (unstable) periodic orbit. But for greater values of $\hat{p}$ we were able to numerically detect an unstable $T$-periodic orbit (dashed lines in Figure \ref{b-r-orbitaTperiodicaInstabile} represents the trajectories backward in time). This $T$-periodic orbit disappears when $p$ becomes too large and the trajectories approach the $T$-periodic orbit $(N_-(t),0)$ as $t\to + \infty$.

Finally, we can numerically compute the eigenvalues of the canonical fundamental matrix varying $\hat{p}$ and find the critical values distinguishing between the three cases of Figure \ref{threeCases}. We approximately obtain the following thresholds,
\begin{itemize}
\item[(a)] $0<\hat{p}<1.21$,
\item[(b)] $1.21<\hat{p}<62.026$,
\item[(c)] $\hat{p}>62.026$,
\end{itemize}
which are in agreement with the critical bifurcation values of the non-periodic system studied in \cite{buffoni2011}. We can also numerically compute the quantity $\alpha$ defined in Theorem \ref{alphaH} when $\hat{p}=4.5$ (corresponding to Figure \ref{b-r-orbitaTperiodicaInstabile}), obtaining a positive value, in agreement with the analytical result.

\begin{figure}
\begin{center}
\subfigure[$\hat{p}=1$\label{a-t-estinzionePredatore}]{\includegraphics[width=.45\columnwidth,trim={3.8cm 8cm 4cm 8cm},clip]{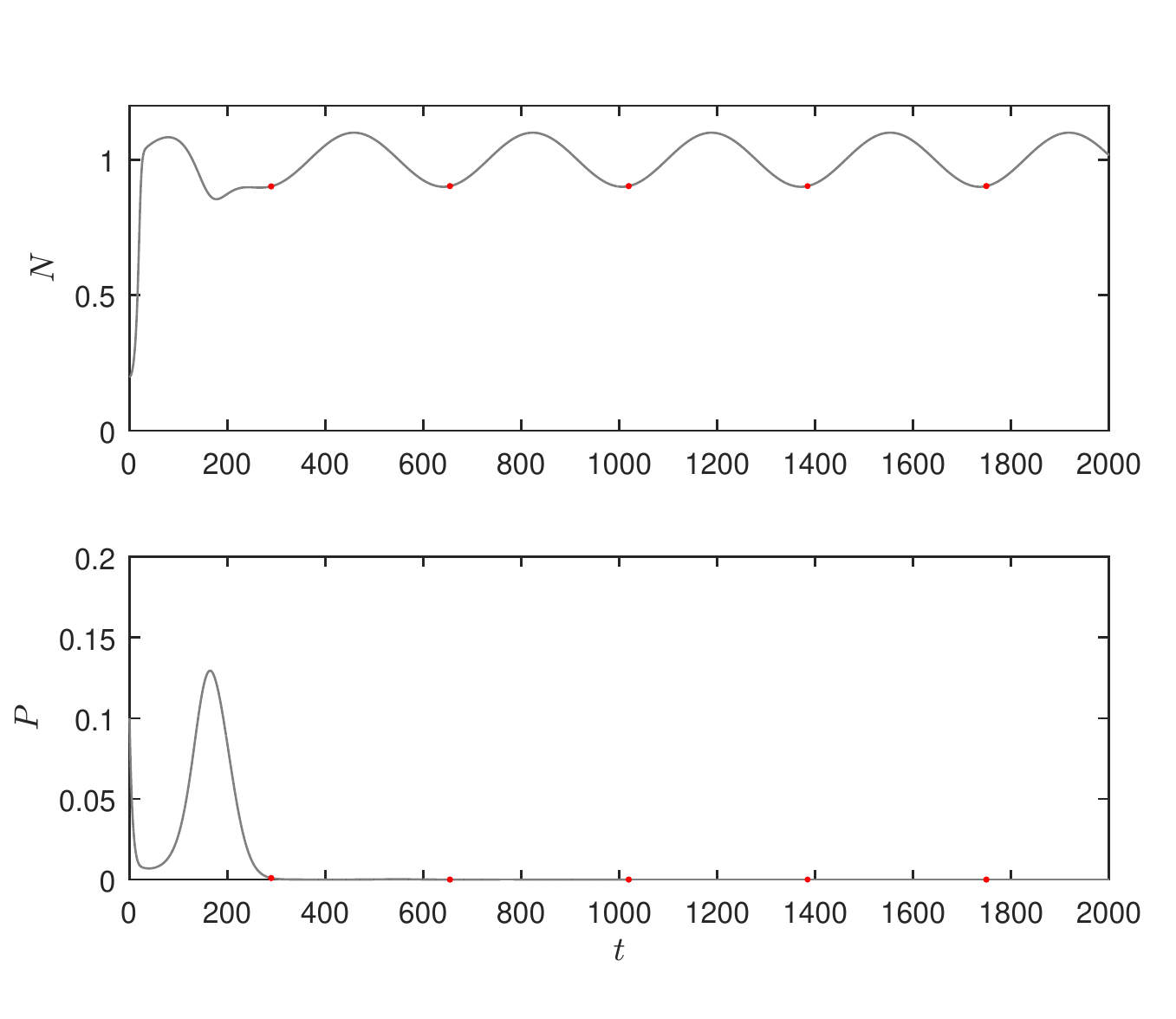}}
\hspace{-0.3cm}
\subfigure[$\hat{p}=1.3$\label{b-t-orbitaTperiodica}]{\includegraphics[width=.45\columnwidth,trim={3.8cm 8cm 4cm 8cm},clip]{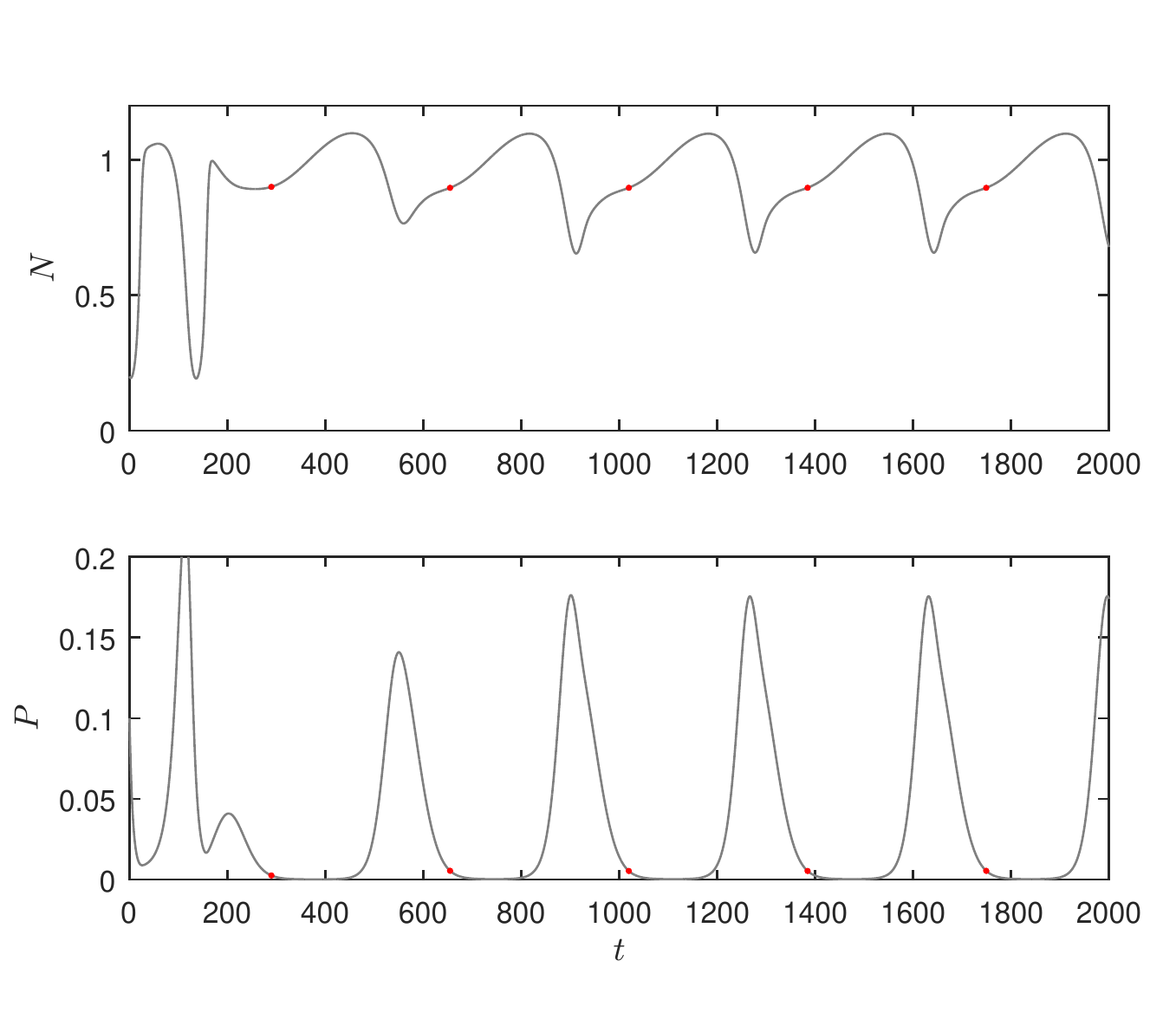}}\\
\subfigure[$\hat{p}=1.55$\label{b-t-orbita3Tperiodica}]{\includegraphics[width=.45\columnwidth,trim={3.8cm 8cm 4cm 8cm},clip]{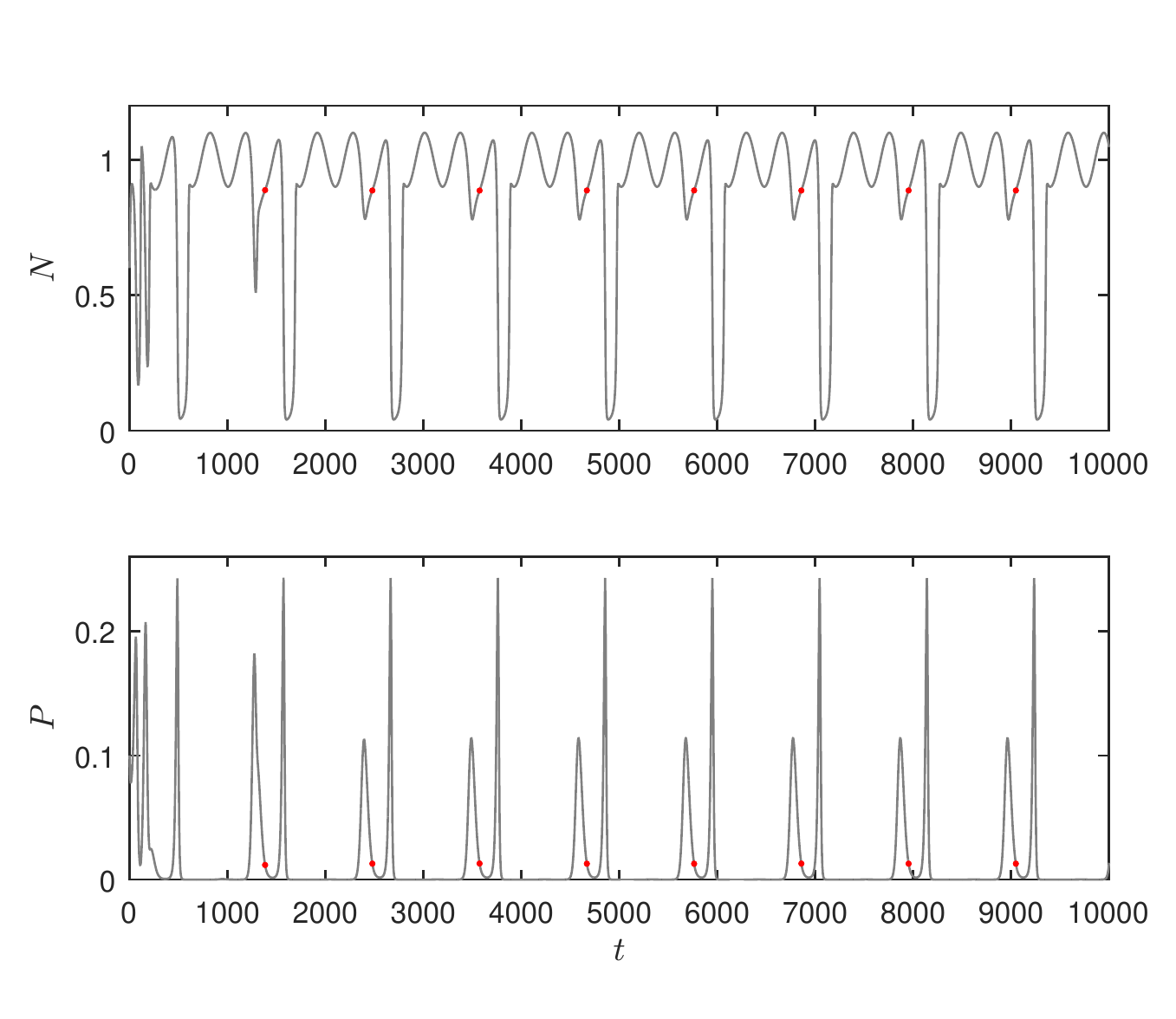}}
\hspace{-0.3cm}
\subfigure[$\hat{p}=1.8$\label{c-t-estinzionetotale}]{\includegraphics[width=.45\columnwidth,trim={3.8cm 8cm 4cm 8cm},clip]{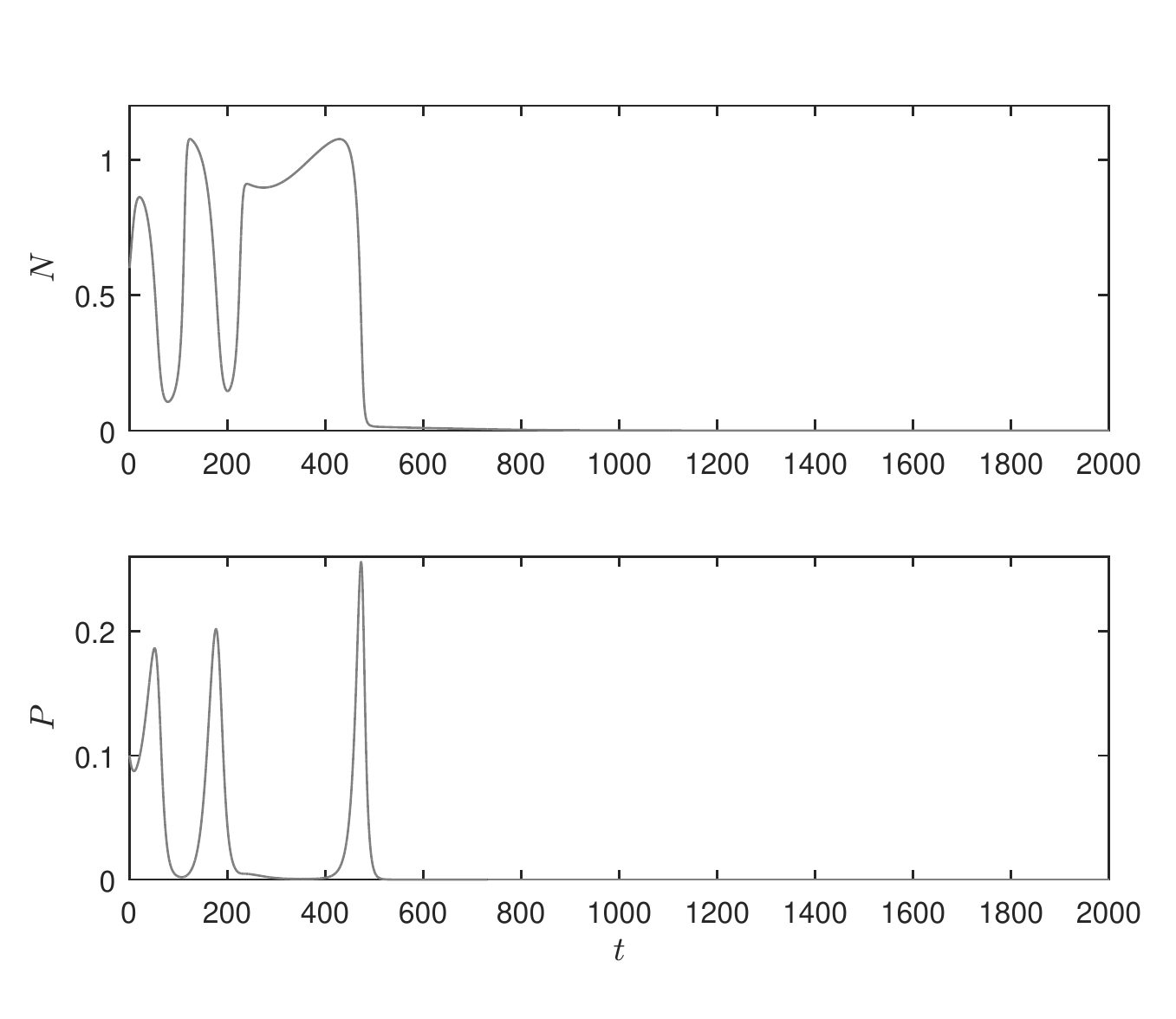}}\\
\caption{Prey and predator abundances over time for different values of $\hat{p}$.}\label{pp_time}
\end{center}
\end{figure}

\begin{figure}
\begin{center}
\subfigure[$\hat{p}=1.3$\label{b-r-orbitaTperiodica}]{\includegraphics[width=.45\columnwidth,trim={3.8cm 8cm 4cm 8cm},clip]{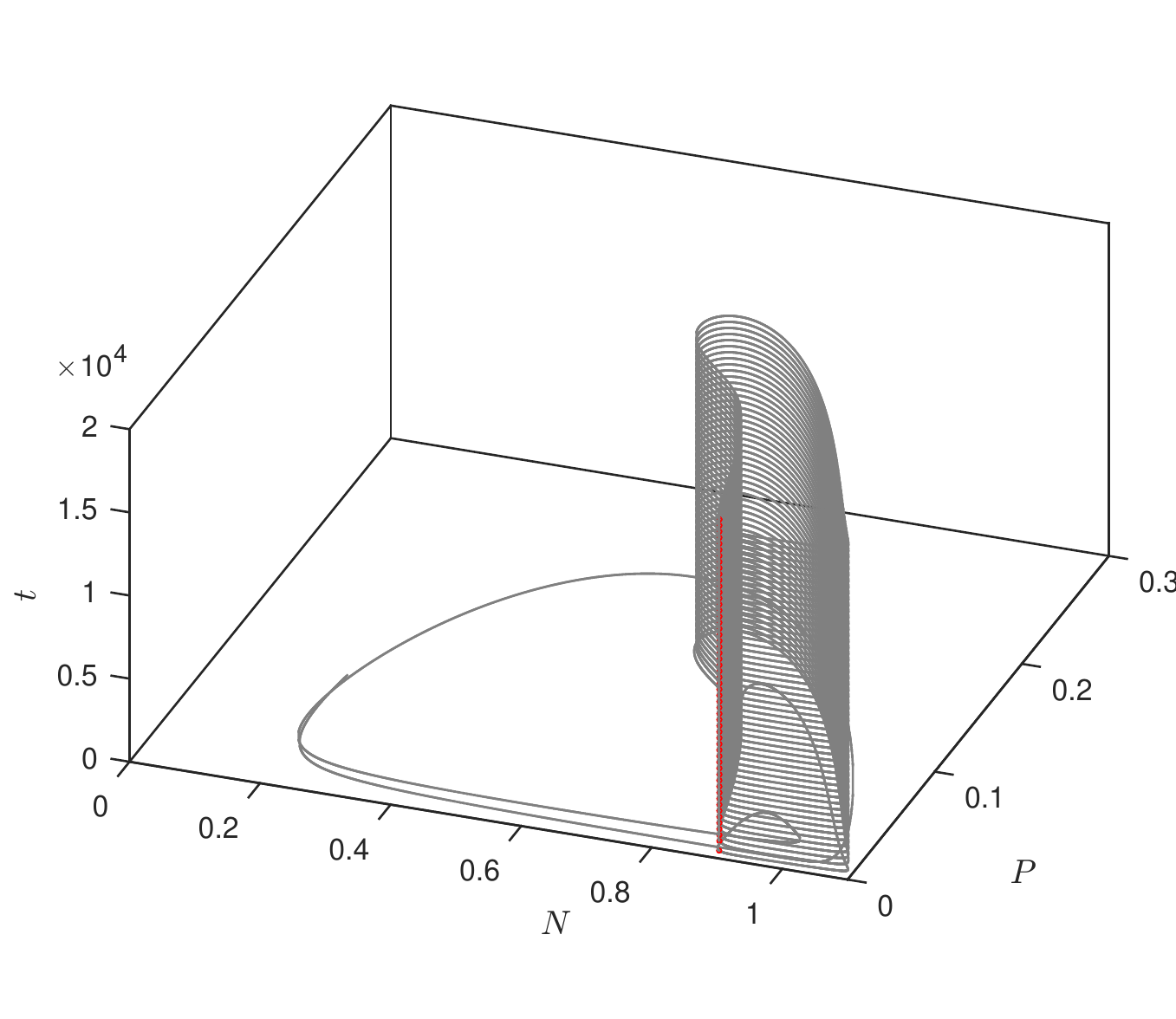}}
\hspace{-0.3cm}
\subfigure[$\hat{p}=1.55$\label{b-r-orbita3Tperiodica}]{\includegraphics[width=.45\columnwidth,trim={3.8cm 8cm 4cm 8cm},clip]{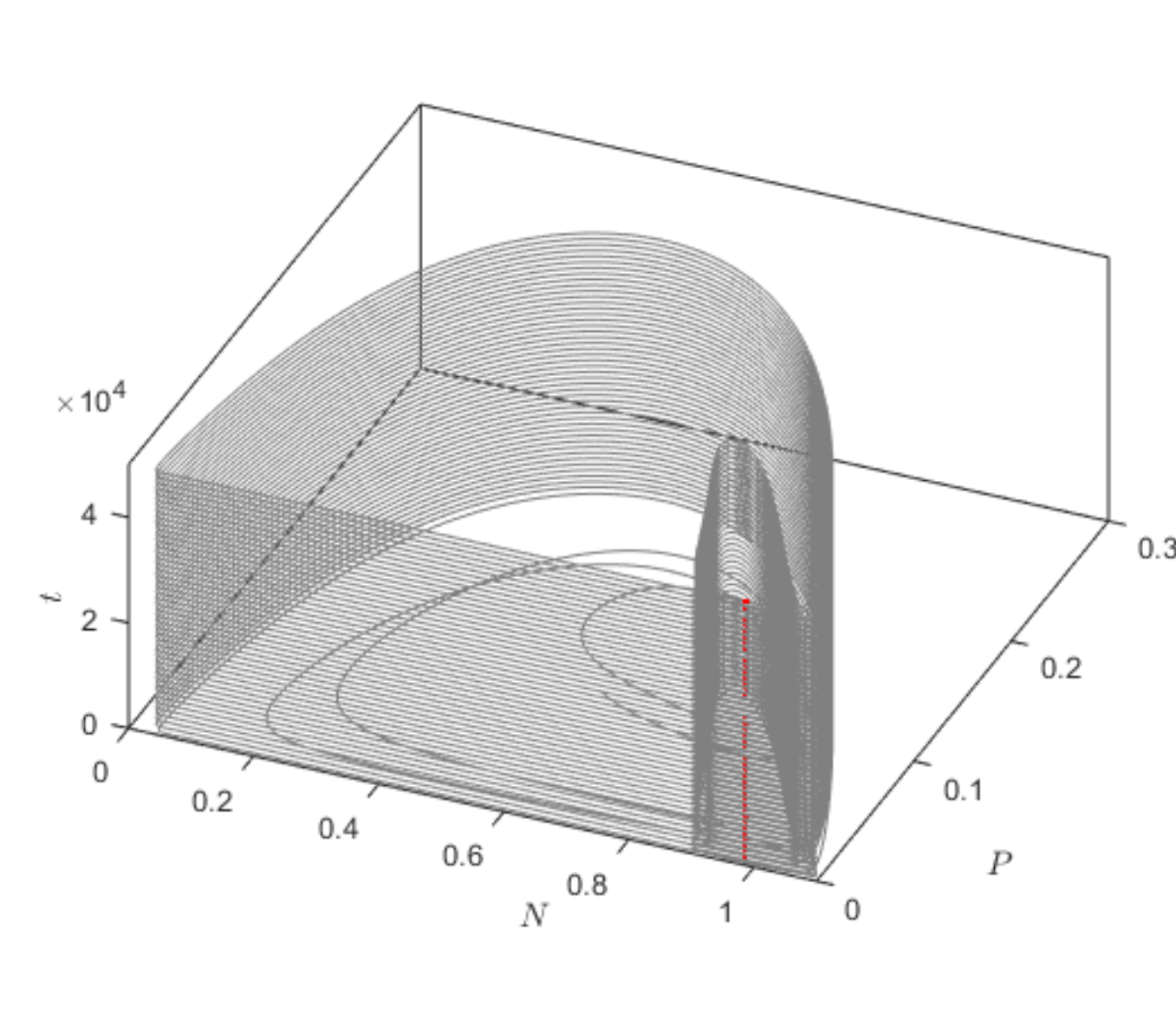}}\\
\caption{Trajectory in the phase-space $(N,P,t)$ for two different value of $\hat{p}$. Red dots indicate the Ponicar\'e Map of period (a) $T$ and (b) $3T$.}\label{pp_cy}
\end{center}
\end{figure}

\begin{figure}
\begin{center}
\includegraphics[width=.8\columnwidth,trim={3.8cm 8cm 4cm 8cm},clip]{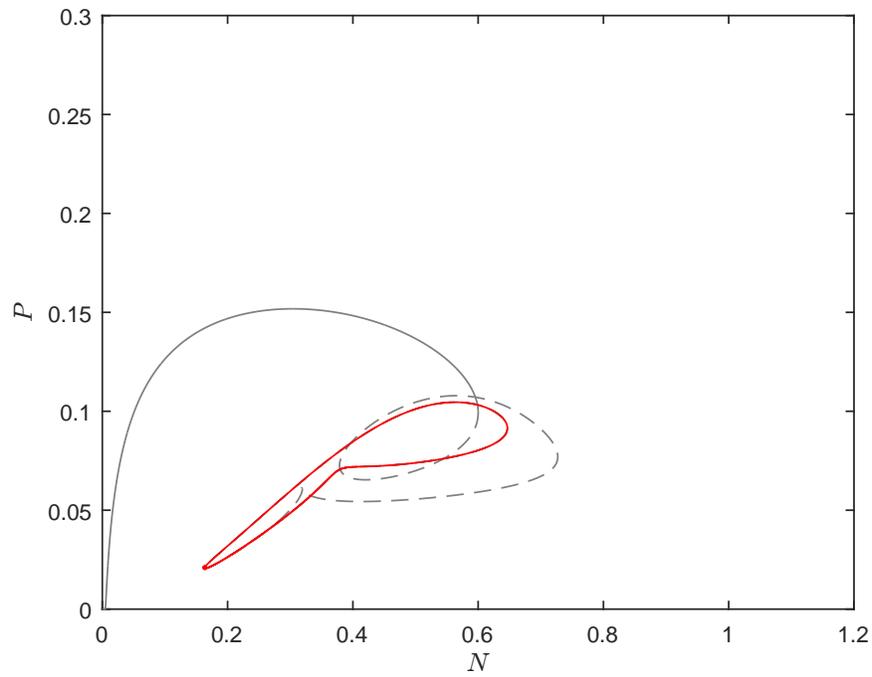}
\caption{Trajectory in the phase-space $(N,P)$ for $\hat{p}=4.5$. Solid and dashed lines denote the trajectories computed forward and backward in time, respectively. The unstable $T$-periodic orbit is marked in red.}\label{b-r-orbitaTperiodicaInstabile}
\end{center}
\end{figure}

\subsection{The Leslie--Gower model with strong Allee effect}
We consider here the periodic Leslie--Gower model with string Allee effect 
\begin{equation}\label{mLG_PM}
\begin{cases}
\dot{N}=r(t)N(N-k(t))\left(1-\dfrac N {K(t)} \right)-\dfrac{b(t)NP}{1+h(t)N+p(t)P},\\
\dot{P}=P\left(a(t)-\dfrac{P}{N+n(t)} \right),
\end{cases}
\end{equation}
which refers to the autonomous one treated in \cite{pal2014bifurcation}, but it is formulated slightly different from system \eqref{mLG} in the previous section. We keep this formulation to compare the results and to use the same parameter set. 

The periodic coefficients $r, \;K, \; b, \; h,\; p,\; a,\; n$ are again defined by \eqref{pc_fav}, while $k$ follows \eqref{pc_unfav}; the parameters values used in the numerical simulations are listed in Table \ref{tab:parsim}, and we consider $s=0.2$. 

In Figure \eqref{LG1} we can see in the phase space $(N,P)$ the trajectories starting from different initial conditions. They can converge to the stable coexistence periodic cycle or to the one in which the prey is extinct; those stable periodic orbits are enlighten in red. In blu we mark the unstable ones, which are the total extinction and two non-coexistence ones in which the predator is extinct. Our conjecture, observing the picture and from the theoretical study, is that another unstable periodic orbit is present, which separates two stable ones. However, it cannot be numerically detected (since we should start on the stable manifold), while the other semi-trivial ones have been detected backward in time. 

Furthermore, increasing the value of the variability scaling factor $s$ in \eqref{pc_fav} and \eqref{pc_unfav} (in order to explore how the variability affects the feature of the system), the coexistence periodic solution becomes wider, it collides with the unstable one and then they disappear.

\begin{table}[h!]
\centering
\begin{tabular}{ccccccccc}
\toprule
$\hat{r}$& $\hat{k}$& $\hat{K}$& $\hat{b}$& $\hat{h}$& $\hat{p}$& $\hat{a}$ & $\hat{n}$ & $T$\\
\midrule
0.4& 2& 12&0.25& 0.375 &  0.175 & 1.5 & 0.1& 365\\
\bottomrule
\end{tabular}
\caption{List of parameters values used in the numerical simulations.}\label{tab:parsim}
\end{table}

\begin{figure}
\begin{center}
\includegraphics[width=.9\columnwidth,trim={3.8cm 8cm 4cm 8cm},clip]{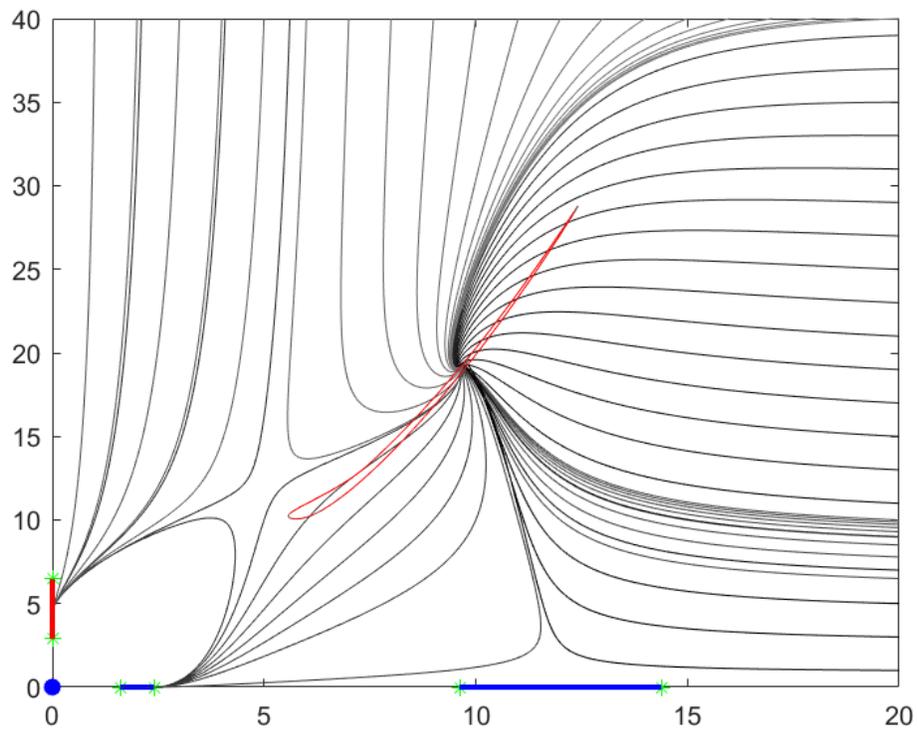}
\caption{Trajectory in the phase-space $(N,P)$ for $s=0.2$. Black lines denote the trajectories. The stable $T$-periodic orbits are marked in red, while in blue we denote the unstable ones. Depending on the initial conditions, trajectories converge to a coexistence periodic orbit, or to the one in which the prey is extinct.}\label{LG1}
\end{center}
\end{figure}


\section{Concluding remarks}\label{Sec:Concl}

In this paper we have studied a general predator--prey model with seasonality in which we take into account an Allee effect on the prey growth. The seasonality is described by periodic coefficients appearing in the mathematical model. The keynote point of the paper is that the theoretical results are obtained for a general class of models, only on the basis of some properties determining the shape of the prey growth function and of the functional response. Furthermore, both cases of weak and strong Allee effects are analyzed. When a weak Allee effect is considered, we prove extinction when the basic reproduction number $R_0<1$, persistence when $R_0>1$, and the existence of a periodic solution when $R_0>1$. The results are obtained exploiting the same techniques used in \cite{garrione2016}, but some preliminary steps are needed in order to obtain crucial properties of the systems in absence of predators with a non-monotonic prey growth function. In the strong Allee effect case, in order to prove the main theorem, an auxiliary result (see Theorem \ref{strong_preyonly}) about the existence of two non-trivial periodic solutions $N_\pm^*(t)$ in the case of a seasonally dependent model for the evolution of one species is obtained. This result generalizes analogous one stated in \cite{rizaner12}.  Thanks to this auxiliary result, we are able to prove extinction of the predators if $\lambda^-_2<\lambda^+_2<1$ and the existence of a nontrivial periodic solution if $\lambda^-_2<1<\lambda^+_2$ where 
$$\displaystyle \lambda^\pm_2=\int_0^T ( \gamma(t)f(t, N_\pm^*(t),0)-\delta_1(t)) dt,$$
being $T$ the period, $\gamma,\; \delta_1$ the conversion factor and the mortality rate of predators respectively, and $N_\pm^*(t)$ the $T$-periodic orbits in absence of predators. Note that now in the predators free line, the solutions can either tend to zero or to $N_+^*(t)$ and hence the $R_0$ associated to this periodic solution which is greater than one if $\lambda^+_2>1$ is not enough to control the dynamics when the number of predators is small.

Finally, we point out that periodic predator--prey models of Leslie--Gower type can be treated using the same techniques. In this case we always have $\lambda^\pm_2>1$ but the dynamics are quite different of the previous case as there exists a semi-trivial periodic solution in the $P$ axis.

Further research directions arise at this point: first of all, it is important to characterize the stability properties of the $T$-periodic orbits. Moreover, in the case $1<\lambda^-_2<\lambda^+_2$, numerical simulations lead us to conjecture that the only feature is the predator extinction, but this is still an open problem that we want to address in a future work. In the Leslie--Gower system, the behavior seems to be different, since we observe in the numerical results the presence of a $T$-periodic orbit. The difference can be related to the presence of the additional non-coexistence $T$-periodic orbit (prey extinction). 

From the numerical simulations we have also seen the presence of periodic orbits with different periods, and the systems seem to show chaotic orbits when the parameter $\hat{p}$ increases in a certain range. Then, after a threshold value, the non-periodic orbits disappear and we only detect the unstable $T$-periodic one. This phenomenon can be related to the formation of a P-to-P connection between the periodic orbits. We plan to investigate it in the future.


\bigskip
{\bf{Acknowledgment}}\\
C.R. and C.S. were supported by FCT-Funda\c{c}\~{a}o para a Ci\^encia e Tecnologia, in the framework of the project UID/MAT/04561/2013. Support by INdAM-GNFM is also gratefully acknowledged by C.S.

\newpage 
\bibliographystyle{elsarticle-harv}
\bibliography{bibliography}
\end{document}